\documentclass{m2an}

\usepackage{amsmath,amsfonts,amssymb,mathtools}
\usepackage{graphicx,subfigure}
\usepackage{braket,dsfont,url,enumerate}

\usepackage{pdfpages} % to import PDF pages
\usepackage{cmap}
\usepackage{hyperref}
\usepackage{color,dsfont,braket}
\usepackage{bbm}
\usepackage[headings]{fullpage}
\usepackage{fullpage}
\usepackage{fancyhdr}
\usepackage{times}
\usepackage{cmap}
\usepackage{pgf}
\usepackage{comment}
\usepackage{todonotes}

\usepackage{tikz}
\usepackage{pgfplotstable}

\renewcommand{\u}{\mathbf{u}}
\newcommand{\uu}{\mathbf{u}}

\newcommand{\ww}{\mathbf{w}}
\newcommand{\vv}{\mathbf{v}}
\newcommand{\oo}{\mathbf{0}}

\newcommand{\g}{\mathbf{g}}

\newcommand{\R}{\mathbb{R}}

\newcommand{\N}{\mathbb{N}}

\newcommand{\Th}{\mathcal{T}_h}

\renewcommand{\div}{\text{div }}

\newcommand{\Om}{\Omega}

\newcommand{\IOprod}[1]{\left( #1 \right)_{I\times \Omega}}
\newcommand{\ImOprod}[1]{\left( #1 \right)_{I_m\times \Omega}}
\newcommand{\IOpair}[1]{\left\langle #1 \right\rangle_{I\times \Omega}}
\newcommand{\psikh}{\psi_{kh}}
\renewcommand{\textforall}{\text{for all }}

\newcommand{\revB}[1]{\textcolor{black}{#1}}

\definecolor{darkred}{rgb}{.7,0,0}

\definecolor{green}{rgb}{0,0.7,0}

\DeclarePairedDelimiter{\norm}{\lVert}{\rVert}

\newcommand{\na}{\nabla}
\newcommand{\pa}{\partial}

\newcommand{\IOm}{I\times \Om}
\newcommand{\ImOm}{I_m\times \Om}

\newcommand{\curl}{\operatorname{curl}}
\newcommand{\Curl}{\operatorname{\mathbf{curl}}}
\newcommand{\vertiii}[1]{{\left\vert\kern-0.25ex\left\vert\kern-0.25ex\left\vert #1
    \right\vert\kern-0.25ex\right\vert\kern-0.25ex\right\vert}}

\newcommand{\half}{\frac{1}{2}}

\newcommand{\mce}{\mathcal{E}_h}

\newcommand{\mct}{\mathcal{T}_h}

\newcommand{\avg}[1]{\left\{\hspace{-0.055in}\left\{#1\right\}\hspace{-0.055in}\right\}}

\newcommand{\jump}[1]{\left[\hspace{-0.035in}\left[#1\right]\hspace{-0.035in}\right]}

\newcommand{\Oprod}[1]{\left(#1\right)_{\Omega}}
\newcommand{\Opair}[1]{\left\langle#1\right\rangle_{\Omega}}

\newcommand{\V}{\mathbf{V}}
\renewcommand{\H}{\mathbf{H}}
\newcommand{\nn}{\mathbf{n}}
\newcommand{\VV}{\mathbb{V}}
\newcommand{\WW}{\mathbb{W}}
\newcommand{\XX}{\mathbb{X}}
\newcommand{\YY}{\mathbb{Y}}

\definecolor{darkred}{rgb}{.7,0,0}
\definecolor{green}{rgb}{0,0.7,0}

\newtheorem{theorem}{Theorem}[section]
\newtheorem{lemma}[theorem]{Lemma}
\newtheorem{proposition}[theorem]{Proposition}
\newtheorem{corollary}[theorem]{Corollary}

\theoremstyle{definition}
\newtheorem{remark}[theorem]{Remark}
\newtheorem{assumption}{Assumption}

\usepackage{zref-clever}
\zcsetup{cap,countertype={subsection=subsection}}

\zcRefTypeSetup{assumption}{Name-sg = Assumption, name-sg=assumption,Name-pl=Assumptions,name-pl=assumptions}
\zcRefTypeSetup{remark}{Name-sg = Remark, name-sg=remark,Name-pl=Remarks,name-pl=remarks}
\zcRefTypeSetup{proposition}{Name-sg = Proposition,Name-pl=Propositions}
\zcRefTypeSetup{corollary}{Name-sg =Corollary,Name-pl=Corollaries}
\zcRefTypeSetup{subsection}{Name-sg=Subsection,Name-pl=Subsections}

\AddToHook{env/proposition/begin}{\zcsetup{countertype={theorem=proposition}}}
\AddToHook{env/lemma/begin}{\zcsetup{countertype={theorem=lemma}}}
\AddToHook{env/assumption/begin}{\zcsetup{countertype={theorem=assumption}}}
\AddToHook{env/remark/begin}{\zcsetup{countertype={theorem=remark}}}
\AddToHook{env/corollary/begin}{\zcsetup{countertype={theorem=corollary}}}

\newcommand{\Cref}[1]{\zcref{#1}}


\begin{document}
\title{Fully discrete error analysis of finite element discretizations of time-dependent
        Stokes equations in a stream-function formulation}

\author{Dmitriy Leykekhman}\address{Department of Mathematics,
               University of Connecticut,
              Storrs,
            CT~06269, USA (\email{dmitriy.leykekhman@uconn.edu}). }
\author{Boris Vexler}\address{Chair of Optimal Control, Department for Mathematics, School of Computation, Information and Technology, Technical University of Munich, Boltzmannstra{\ss}e 3, 85748 Garching b. Munich, Germany
(\email{vexler@tum.de, wagnerja@cit.tum.de}). }
\author{Jakob Wagner}
\sameaddress{2}
% \address{Chair of Optimal Control, Department for Mathematics, School of Computation, Information and Technology, Technical University of Munich, Boltzmannstra{\ss}e 3, 85748 Garching b. Munich, Germany
% (wagnerja@cit.tum.de). }

\begin{abstract}
In this paper we establish best-approximation-type error estimates for fully discrete Galerkin solutions of 
the time-dependent Stokes problem
using the stream-function formulation. The resulting equations involve the spatial biharmonic operator and the time derivative of the Laplacian.
For the time discretization we use the discontinuous Galerkin method
of arbitrary degree,
whereas we  present the space discretization in a general framework. This makes our result applicable for a wide variety of space discretization methods, provided some Galerkin orthogonality conditions are satisfied.
As an example, conformal $C^1$ and $C^0$  interior penalty methods are covered by our analysis. The results do not require any additional regularity assumptions beyond the natural regularity given by the domain and data and can be used for optimal control problems. A numerical example illustrates the theoretical findings. 
\end{abstract}


%-------------------------------------------------------
\maketitle

% \pagestyle{myheadings}
% \markboth{DMITRIY LEYKEKHMAN, BORIS VEXLER AND JAKOB WAGNER}{Fully discrete analysis of time dependent Stokes via stream-function formulation }
% 
% \maketitle
% 
% \renewcommand{\thefootnote}{\fnsymbol{footnote}}
% \footnotetext[2]{Department of Mathematics,
%                University of Connecticut,
%               Storrs,
%               CT~06269, USA (dmitriy.leykekhman@uconn.edu). }
% 
% \footnotetext[3]{Chair of Optimal Control, Department for Mathematics, School of Computation, Information and Technology, Technical University of Munich, Boltzmannstra{\ss}e 3, 85748 Garching b. Munich, Germany
% (vexler@tum.de, wagnerja@cit.tum.de). }
% 
% \renewcommand{\thefootnote}{\arabic{footnote}}
% 
%\subjclass{65N30,65N15}
%\date{June 1, 2012}
%\dedicatory{}
%\keywords{parabolic, fully discrete, discontinuous Galerkin, finite elements, optimal error estimates, optimal control, state constrained}

%%%%%%%%%%%%%%%%%%%%%%%%%%%%%%%%%%%%%%%%%%%%%%%%%%%%%%%%%%%%%%%%%%%%%%%%%%%%%%%%%%%%%%%%%%%%%%%55%%%%%%%%
\section{Introduction}

In this paper we consider the following time-dependent Stokes problem with no-slip boundary conditions, 
\begin{equation}\label{eq:transient:Stokes}
    \begin{aligned}
	\partial_t\u-\Delta \u+ \nabla p &= \g \quad &\text{in }I\times\Omega, \\
	\nabla \cdot \u &= 0 \quad &\text{in } I\times\Omega, \\
	\u &= \oo \quad  &\text{on } I\times\partial \Omega, \\
	\u(0) &= \u_0 &\text{in } \Omega.
    \end{aligned}
\end{equation}
We assume that $\Omega\subset \mathbb{R}^2$  is a bounded, polygonal and simply connected 
domain, $T>0$ and $I=(0,T]$.
The finite element approximation of Stokes and Navier-Stokes problems is a mature and well-researched topic due to its importance for understanding and computing fluid problems. The literature is immense, we would only like to mention the monographs \cite{Glowinski_2003, 1986Girault, temam_navier_stokes_2001}.  

The numerical, and in particular finite element approximation of the problem, has a long history and the research is still very active. There are many challenges one faces in order to construct  proper finite element spaces. It is well-known that the spaces for the discrete velocity $\u_h$ and the discrete pressure $p_h$ cannot be chosen arbitrarily and the corresponding finite element spaces must satisfy the {\it inf-sup} condition. In addition, the continuity equation, i.e.,  $\nabla \cdot \u = 0$, is not trivial to satisfy on the discrete level pointwise, and usually requires either high order methods \cite{AinsworthM_ParkerC_2021, FalkR_NeilanM_2013, ScottR_VogeliusM_1985} or very special constructions  \cite{ GuzmanJ_NeilanM_2014b, GuzmanJ_NeilanM_2014}. The time discretization poses new challenges. For example, on non-matching meshes, a discretely divergence free function on 
  the previous mesh is not discretely divergence free on the next mesh and needs to be projected, which
leads to oscillations in the pressure.
All these challenges are well-known and are addressed in more or less satisfactory fashion and constitute the standard part of the finite element theory. 
Recently, the  finite element theory and practices have been viewed in a new light and many standard finite element approximation schemes faced some criticism. The main objection of  many standard finite element methods, like Taylor-Hood or MINI elements, are the facts that the discrete pressure and velocity are coupled and the error estimates for the velocity contain a pressure term. As a result, in some situations, the perturbations of the source term by a gradient field is not completely absorbed by the discrete pressure alone and may lead to a large error for the computed velocity \cite{JohnV_LinkeA_MerdonC_NeilanM_RebholzL_2017}. This observation has recently led to the development of various  {\it pressure-robust} methods \cite{AhmedN_BarrenecheaG_BurmanE_GuzmanJ_LinkeA_MerdonC_2021, AhmedN_LinkeA_MerdonC_2018, LinkeA_MerdonC_NeilanM_2020}. 
%Recently, A. Linke also raised a question of gradient consistent methods \cite{LinkeA_2024}.  

An old two-dimensional approach that addresses all the above-mentioned points  is the stream-function method where the velocity $\u$ is taken as vectorial curl of a scalar function $\psi$
\begin{equation}\label{eq:u_psi_relationship}
\u = \Curl(\psi):=(\partial_2 \psi, -\partial_1 \psi)^T.
\end{equation}
Its dual operator is given by the scalar curl of a vector field $\vv$
$$
  \curl(\vv) = \partial_2 \vv_1 - \partial_1 \vv_2.
$$
For $\vv \in H^1_0(\Om)^2, \psi \in H^1_0(\Om)$, there hold the following identities, to be understood in 
$H^{-1}(\Om)^2$ and $H^{-1}(\Om)$, respectively:
\begin{align}
  \Curl(\curl(\vv)) &= (\partial_2^2 \vv_1 - \partial_2 \partial_1 \vv_2,-\partial_1 \partial_2 \vv_1+\partial_1^2 \vv_2)^T = \Delta \vv - \nabla (\div \vv)
  \label{eq:Curlcurl}\\
  \curl(\Curl(\psi)) &= \Delta \psi. \label{eq:curlCurl}
\end{align}
For velocities $\u$ satisfying the form \eqref{eq:u_psi_relationship}, it clearly holds $\div\u=0$,
and taking formally the negative scalar curl of the equation
\eqref{eq:transient:Stokes}, we find that the scalar function $\psi$ satisfies the following fourth order  equation with $f = -\curl(\g)$
\begin{equation}\label{eq:transient:Fourth}
    \begin{aligned}
	-\partial_t\Delta\psi+\Delta^2 \psi &= f \quad &\text{in }I\times\Omega, \\
	\psi=\partial_n\psi &=  0 \quad  &\text{on } I\times\partial \Omega, \\
	\psi(0) &= \psi_0 &\text{in } \Omega,
    \end{aligned}
\end{equation}
where $\psi_0$ is the solution to 
\begin{equation}\label{eq:initial_psi0}
    \begin{aligned}
	-\Delta \psi_0 &= -\curl(\u_0) \quad &\text{in }\Omega, \\
	\psi_0 &=  0 \quad  &\text{on } \partial \Omega.
    \end{aligned}
\end{equation}
For simply connected $\Omega$, the equations  \eqref{eq:transient:Stokes} and \eqref{eq:transient:Fourth} 
are equivalent, for a proof, we refer to \cite{GuermondJL_QuartapelleL_1994}.
 Although the stationary problem is well investigated, (cf. Parts 4 and 5 in \cite{GunzburgerM_2012}),
it seems there is no clean theory of fully discrete finite element error estimates for Galerkin methods for the time-dependent problem.
For example, a relatively recent paper \cite{AdakD_MoraD_NatarajanS_SilgadoA_2021} treated backward Euler for virtual elements and derived some fully discrete error estimates in the $L^\infty(H^1)$-norm. 
However, in order to achieve these estimates, the authors rely on regularity in time for spatial error estimates
and had to impose overly strict assumptions on 
$\partial_{tt} \psi$ and $\g \in L^\infty(I;H^{r-1}(\Om)^2)$ for some $r>\frac12$.
% However, besides treating only lowest order discretization the error analysis presented there is rather messy and involves unnecessary terms,
% like time derivatives for the estimation of the spatial error, as well as 
% artificial assumptions on $\partial_{tt} \psi$ and $\g \in L^\infty(I;H^{r-1}(\Om)^2)$ for some $r>\frac12$.
Another recent paper \cite{KimQuasiGeostrophic2025} discusses discretization of a similar equation with $C^1$ conforming elements in space and backward Euler in time. %It requires also similar strong assumptions on the data as \cite{AdakD_MoraD_NatarajanS_SilgadoA_2021}.

In this paper, we take a first step in the direction of developing an elegant theory for the fully discrete error estimates. %similarly to the one for the second order parabolic problems \cite{LeykekhmanD_VexlerB_2016a, LeykekhmanD_VexlerB_2017a}. 
\revB{The numerical analysis for discontinuous Galerkin discretizations in time for parabolic problems goes back to \cite{jamet_galerkin_type_1978,ErikssonJohnsonThomee:1985}, see also \cite{LarssoThomeeWahlbin:1998}.
Based on the ideas of (fully) discrete maximal parabolic regularity for this class of methods, see, \cite{LeykekhmanD_VexlerB_2017a,LeykekhmanVexler:2018SINUM}, best-approximation-type error estimates in various norms were developed for linear parabolic \cite{LeykekhmanD_VexlerB_2016a,LeykekhmanVexler:2017SINUM_Gradient}, transient Stokes \cite{Behringer_Leykekhman_Vexler_2022,LeykekhmanVexler:2024Calcolo}, and Navier-Stokes equations \cite{VexlerWagner:2024M2AN}.}
We establish a best-approximation-type error estimate  for discontinuous Galerkin discretizations in time of arbitrary degree and finite element discretization in space  (that include conformal $C^1$ elements as well as $C^0$ interior penalty methods) of \eqref{eq:transient:Fourth}, namely
\begin{equation}\label{eq:psi:best_approx_intro}
      \norm{\na(\psi-\psi_{{k}h})}_{L^2(I\times\Om)} 
      \leq C  \Big(\norm{\na(\psi-{\chi})}_{L^2(I\times\Om)} 
          +  \norm{\na(\psi - R_h\psi)}_{L^2(I\times\Om)}
      +   \norm{\na(\psi - \pi_k\psi)}_{L^2(I\times\Om)} \Big), %
\end{equation}
where $\psi_{kh}$ is the fully  discrete Galerkin solution for \eqref{eq:transient:Fourth}, $R_h$ is the Galerkin projection of the stationary biharmonic problem, $\pi_k$ is a certain time projection, and  $\chi$ is an arbitrary element from the finite dimensional space. 
Under low regularity assumptions, i.e., $f \in L^2(I;H^{-1}(\Om))$ and $\psi_0 \in H^2_0(\Om)$, 
using the approximation properties of the above-mentioned projections, we can establish the following  optimal convergence result 
\begin{equation}\label{eq:psi:rates}
    \norm{\na(\psi-\psi_{kh})}_{L^2(I\times\Om)} 
    \leq C  (k+h^2 )\Big(
    %\|(\curl(\f)\|_{L^{2}(I;L^2(\Omega))}+\|\na\psi_0\|_{L^2(\Om)}
     \|f\|_{L^{2}(I;H^{-1}(\Omega))}+\|\psi_0\|_{H^2_0(\Om)}
  \Big) 
\end{equation}
on convex polygonal domains.
The minimal assumptions we put on the data are satisfied by the usual assumptions on data for the time-dependent Stokes problem, namely $\g \in L^2(I;\H)$ and $\uu_0 \in H^1_0(\Om)^2$ with $\nabla \cdot \uu_0 = 0$.
For a precise introduction of the space $\H$, see \Cref{subsec:stokes_theory}.
Indeed we will argue later, see \eqref{eq:curl_isom} and \Cref{lemm:initial_data_stability},
that under those assumptions it holds $f \in L^2(I;H^{-1}(\Om))$ and $\psi_0 \in H^2_0(\Om)$.
As a result, taking $\u_{kh}=\Curl(\psi_{kh})$ we obtain for the velocity solution 
of \eqref{eq:transient:Stokes} 
\begin{equation}\label{eq:psi:rates velocity}
    \begin{aligned}
        \norm{\u-\u_{kh}}_{L^2(I\times \Om)}
    &=\norm{\Curl(\psi-\psi_{kh})}_{L^2(I\times \Om)}
    = \norm{\na(\psi-\psi_{kh})}_{L^2(I\times \Om)}\\
    & \leq C  (k+h^2)\left(\norm{\g}_{L^2(I\times \Om)} + \norm{\uu_0}_{H^1_0(\Om)}\right).
    \end{aligned}
\end{equation}
    As we have formulated our main result \eqref{eq:psi:best_approx_intro} in a best-approximation-type form, higher order estimates
    are also possible, assuming higher regularity of $\psi$.
    
    This paper presents our initial step in the numerical analysis of stream-function formulations.
    We plan to extend our results to surface flows, which are inherently two-dimensional (see \cite{MR4160329,NeilanWan:2025} for the stationary case), to multiply connected domains, see \cite{GuermondJL_QuartapelleL_1997}, as well as to the Navier–Stokes equations. Moreover, the only natural (low) regularity assumptions we make on the data allow application to optimal control problems, cf. \cite{VexlerMeidner2025}.
    
    The remainder of the paper is structured as follows:
    in \Cref{sec:continuous}, we summarize some results of the continuous Stokes problem,
    and establish some results regarding the associated stationary and time-dependent biharmonic equations.
    In \Cref{sec: discretization} we introduce the space and time discretizations, keeping the 
    discussion of the space discretization very short and general, and establish the first stability result
    for the fully discrete problem.
    We proceed to prove the best-approximation-type error estimate in \Cref{sec:error_analysis}, with 
    our main result being \Cref{thm:error fully discrete}. After this, we discuss the specific orders of 
    convergence for special choices of the space discretization.
    In \Cref{sec:numerical_experiments} we present some numerical experiments illustrating our theoretical findings. In \Cref{sec:appendix}, we discuss the sharpness of our regularity/stability results from \Cref{thm:existence_uniqueness_standard_regularity} by  providing a counterexample, which shows that higher regularity of the time-derivative of the solution cannot be expected.

%%%%%%%%%%%%%%%%%%%%%%%%%%%%%%%%%%%%%%%%%%%%%%%%%%%%%%%%%%%%%%%%%%%%%%%%%%%%%%%%%%%%%%%%%%%%%%%%%%%%%%%%%%%%%

\section{Continuous problem}\label{sec:continuous}

\subsection{Notation}
We will use the standard notation for Lebesgue and Hilbert spaces. Most of our analysis will be $L^2$ based and we will denote the space and space-time inner-products by 
$$
\Oprod{f,g} =\int_\Omega fg\, dx\quad \text{and}\quad (f,g)_{I\times\Omega}=\int_I(f,g)_\Omega dt,
$$
respectively, with the corresponding norms
$$
\|f\|^2_\Omega =\int_\Omega |f|^2 \, dx\quad \text{and}\quad \|f\|^2_{I\times\Omega}=\int_I\|f\|^2_\Omega dt.
$$
When the situation is less clear we will use the standard notation for the norms, $L^2(\Om)$, $H^2(\Om)$ etc. 
We will denote by $\langle\cdot,\cdot\rangle$ a dual pairing between a function space and its dual.

\subsection{Function spaces}

We define the following functional spaces.
$$
\begin{aligned}
H^1_0(\Om) &=\{v\in H^1(\Om)\ |\ v=0\quad \text{on $\partial \Om$}\},\\
H^2_0(\Om) &=\{v\in H^2(\Om)\ |\ v=\partial_nv=0\quad \text{on $\partial \Om$}\},\\
H^{-2}(\Om)&=\left(H^2_0(\Om)\right)^*,\quad H^{-1}(\Om)=\left(H^1_0(\Om)\right)^*.\\
%X&=H^1(I; L^2(\Om))\cap L^2(I; H^2_0(\Om))\\
%Y&= L^2(I; H^2_0(\Om)).
\end{aligned}
$$
\begin{lemma}\label{lemm:h2_norm_equivalence}
  On the function space $H^2_0(\Om)$, 
  the norm $\norm{v}_{H^2_0(\Om)} := \norm{\Delta v}_{L^2(\Om)}$ is equivalent to $\norm{v}_{H^2(\Om)}$.
\end{lemma}
\begin{proof}
The proof follows by integration by parts and Poincaré inequality for the lower order terms.
\end{proof}

    \subsection{Results on the time-dependent Stokes problem}\label{subsec:stokes_theory}

    For a better comparison of the stream-function formulation to the classic formulation of the Stokes problem,
    let us recall some well known results.
    Most commonly, the Stokes equations on the continuous level are discussed in divergence-free spaces,
    to which end we introduce the notations
    \begin{align*}
        \V := \set{\vv \in H^1_0(\Om)^2 | \nabla \cdot \vv = 0}
        \quad \text{ and } \quad 
        \H := \set{\vv \in L^2(\Om)^2| \nabla \cdot \vv = 0 \land \nn\cdot \vv = 0},
    \end{align*}
    where by $\nn \cdot \vv$ we denote the normal trace of $\vv$. It is straightforward to verify the 
    following result
    \begin{lemma}\label{lemm:V_norm_equivalence}
        On the space $H^1_0(\Om)^2$, the following norms are equivalent:
        $$\norm{\vv}_{H^1_0(\Om)} \cong \norm{\na \cdot \vv}_{L^2(\Om)} + \norm{\curl(\vv)}_{L^2(\Om)}.$$
        Especially, $\norm{\curl(\vv)}_{L^2(\Om)}$ induces a norm on $\V$ equivalent to the $H^1_0(\Om)$ norm.
    \end{lemma}
    \begin{proof}
      For $\vv \in C^\infty_0(\Om)^2$ we obtain using the definition of $\nabla \cdot$ and $\curl$, as well as 
      integration by parts
      \begin{align*}
        %\|\vv\|_{H^1_0(\Om)}^2 = \int_\Om \sum_{i,j=1}^2 (\partial_i \vv_j)^2 \, dx
        \|\nabla \cdot \vv\|_{L^2(\Om)}^2 + \|\curl(\vv)\|_{L^2(\Om)}^2
        &= \int_{\Om} (\partial_1 \vv_1 + \partial_2 \vv_2)^2 + (\partial_1 \vv_2 - \partial_2 \vv_1)^2 \, dx\\
        &= \int_\Om \sum_{i,j=1}^2 (\partial_i \vv_j)^2 + 2 \partial_1 \vv_1 \partial_2 \vv_2 - 2 \partial_1 \vv_2 \partial_2 \vv_1 \, dx\\
        &= \int_\Om \nabla \vv : \nabla \vv - 2 \vv_1 \partial_1 \partial_2 \vv_2 + 2 (\partial_2 \partial_1 \vv_2) \vv_1 \, dx\\
        &= \int_\Om \nabla \vv : \nabla \vv \, dx = \|\vv\|_{H^1_0(\Om)}^2.
      \end{align*}
      The statements of the lemma then follow from density of $C^\infty_0(\Om)$ in $H^1_0(\Om)$, and the
      definition of $\V$.
    \end{proof}
    
    Using these spaces, the stationary Stokes operator $A\colon D(A) \to \H$ is given by 
    \begin{equation}\label{eq:def_stokes_operator}
        \Oprod{A \uu,\vv} = \Oprod{\nabla \uu,\nabla \vv} \quad \text{ for all } \vv \in \V,
    \end{equation}
    where $D(A) := \{\vv \in \V: \Delta \vv \in L^2(\Om)^2\}$.
    The operators $\curl, \Curl$ lead to the following isomorphisms between scalar valued Sobolev spaces and 
    divergence-free vector spaces, see \cite[Lemma 2]{GuermondJL_QuartapelleL_1994}:
    \begin{align}
        \Curl\colon& H^1_0(\Om) \to \H, && \text{and}& \Curl\colon& H^2_0(\Om) \to \V,\label{eq:Curl_isom}\\ 
        \curl\colon& \H \to H^{-1}(\Om), && \text{and} & \curl\colon& \V^* \to H^{-2}(\Om).\label{eq:curl_isom}
    \end{align}
    \revB{The isomorphism property $\curl\colon \V^* \to H^{-2}(\Om)$ follows directly from the isomorphism $\Curl\colon H^2_0(\Om) \to \V$ by the fact that $\curl$ is the adjoint operator for $\Curl$.}
    
    The time-dependent Stokes equations are described by the spaces
    \begin{align*}
        \VV & := L^2(I;\V) \cap H^1(I;\V^*) \hookrightarrow C(\bar I;\H),\\
        \WW & := L^2(I;\V \cap H^2(\Om)^2) \cap H^1(I;\H) \hookrightarrow C(\bar I;\V).
    \end{align*}
    The weak formulation of the time-dependent Stokes equations reads:
    Find $\uu \in \VV$ such that $\uu(0) = \uu_0$ and
    \begin{equation}\label{eq:stokes_weak_formulation}
        \IOpair{\partial_t \uu,\vv} + \IOprod{\nabla \uu,\nabla \vv} 
        = \IOpair{\g,\vv} \quad \textforall \vv \in L^2(I;\V),
    \end{equation}
    where $\IOpair{\cdot,\cdot}$ denotes the duality pairing between $L^2(I;\V)$ and $L^2(I;\V^*)$.
    There holds the following result, see 
    \cite[Chapter III, Theorem 1.1 Proposition 1.2]{temam_navier_stokes_2001}:
    \begin{proposition}
        Let $\g \in L^2(I;\V^*)$ and $\uu_0 \in \H$, then \eqref{eq:stokes_weak_formulation}
        has a unique weak solution $\uu \in \VV$ with
        \begin{equation*}
            \norm{\uu}_{\VV} \le C(\norm{\g}_{L^2(I;\V^*)} + \norm{\uu_0}_{\H}).
        \end{equation*}
        If $\g \in L^2(I;L^2(\Om)^2)$ and $\uu_0 \in \V$, then $\uu$ additionally satisfies
        \begin{equation*}
            \norm{\partial_t \uu}_{L^2(I;\H)} + \norm{A \uu}_{L^2(I;\H)} \le
            C(\norm{\g}_{L^2(I;L^2(\Om))} + \norm{\uu_0}_{\V}),
        \end{equation*}
        where $A$ is the Stokes operator defined in \eqref{eq:def_stokes_operator}.
        If moreover $\Omega$ is convex, then $\uu \in \WW$ with a bound
        \begin{equation*}
            \norm{\uu}_{\WW} \le C (\norm{\g}_{L^2(I;L^2(\Om))} + \norm{\uu_0}_{\V}).
        \end{equation*}
    \end{proposition}
    We have already stated, that the initial data $\psi_0$ for the stream-function formulation and 
    the data $\uu_0$ for the standard formulation have to satisfy the relation \eqref{eq:initial_psi0}.
    In the following lemma, we establish that the regularity assumptions posed on $\psi_0$ are actually natural
    and hold true, if the corresponding standard assumptions are posed on $\uu_0$.
    \begin{lemma}\label{lemm:initial_data_stability}
        Let $\uu_0 \in \H$ and $\psi_0$ solve \eqref{eq:initial_psi0}. Then $\psi_0 \in H^1_0(\Om)$, 
        and there holds a bound 
        \begin{equation*}
            \norm{\psi_0}_{H^1_0(\Om)} \le C \norm{\uu_0}_{\H}.
        \end{equation*}
        If furthermore, $\uu_0 \in \V$, then $\psi_0 \in H^2_0(\Om)$ and there holds the bound
        \begin{equation*}
            \norm{\psi_0}_{H^2_0(\Om)} \le C \norm{\uu_0}_{\V}.
        \end{equation*}
    \end{lemma}
    Before we prove the lemma, note that the result $\psi_0 \in H^2_0(\Om)$ only poses a requirement on $\uu_0$
    and we do not require, e.g., convexity of $\Omega$, which would be otherwise a standard assumption to achieve
    $H^2$ regularity.
    \begin{proof}
        For the case $\uu_0 \in \H$ there is nothing to prove, as this is directly given by the weak
        formulation of \eqref{eq:initial_psi0}.
        Hence assume now $\uu_0 \in \V$. 
        Our goal is to show the result by extending the proof of \cite[Lemma 2]{GuermondJL_QuartapelleL_1994}.
        As a result of the identity \eqref{eq:curlCurl}, there holds 
        \begin{equation*}
            \curl(\Curl \psi_0 - \uu_0) = 0.
        \end{equation*}
          By \cite[Lemma 3]{GuermondJL_QuartapelleL_1994} there exists a unique $p \in H^1(\Om)/\R$ such 
          that $\Curl \psi_0 - \uu_0 = \nabla p$.
        Due to the first isomorphism of \eqref{eq:Curl_isom} $\Curl \psi_0 \in \H$ and by the assumption
        $\uu_0 \in \V \subset \H$, it holds $\nabla p \in \H$. The definition of $\H$ then implies 
        $\Delta p = 0$ and $\partial_\nn p = 0$. Hence $p$ is constant and due to $p \in H^1(\Om)/\R$ it 
        actually vanishes. This implies $\Curl \psi_0 = \uu_0 \in \V \subset H^1_0(\Om)^2$.
        By definition of $\Curl$, this implies 
        \begin{equation*}
            \partial_1 \psi_0, \partial_2 \psi_0 \in H^1_0(\Om).
        \end{equation*}
        Hence especially $\psi_0 \in H^2(\Om)$ and 
        $\partial_\nn \psi_0 
        %\nn \cdot \begin{pmatrix} 0\\0 \end{pmatrix}
        =0$ on $\partial \Om$.
        This shows $\psi_0 \in H^2_0(\Om)$.
        The norm bound can then be obtained using \Cref{lemm:h2_norm_equivalence,lemm:V_norm_equivalence},
        which yields
        \begin{equation*}
            \norm{\psi_0}_{H^2_0(\Om)} \cong \norm{\Delta \psi_0}_{L^2(\Om)}
            = \norm{\curl \uu_0}_{L^2(\Om)} \cong \norm{\uu_0}_{\V}.
        \end{equation*}
    \end{proof}

\subsection{Results for stationary biharmonic equation}

Before analyzing the time-dependent problem, we recall some basic results for the biharmonic problem with Dirichlet (clamped) boundary conditions,
\begin{equation}\label{eq:Fourth}
    \begin{aligned}
	\Delta^2 w &= f \quad &\text{in }\Omega, \\
	w=\partial_n w &=  0 \quad  &\text{on } \partial \Omega.
	    \end{aligned}
\end{equation}
The weak formulation for \eqref{eq:Fourth} reads
\begin{equation}\label{eq: weak}
w\in H^2_0(\Om) \quad : \quad	\Oprod{\Delta w, \Delta \varphi} 
= \Opair{f, \varphi} \quad \textforall \varphi\in H^2_0(\Om).
\end{equation}
The following results can be found, e.g., in \cite[Sec. 5.9]{Kondratev_1967}.
\begin{proposition}\label{prop: elliptic}
Let $\Om\subset \mathbb{R}^2$ be a bounded Lipschitz domain and $f\in H^{-2}(\Om)$.
Then \eqref{eq: weak} has a unique weak solution $w\in H^2_0(\Om)$ and there exists a constant $C$ such that 
\begin{equation}\label{eq: H2 regularity}
\|w\|_{H^{2}(\Om)}\le C\|f\|_{H^{-2}(\Om)}.
\end{equation}
If in addition $\Om$ is convex and $f\in H^{-1}(\Om)$, then the weak solution $w$ of  \eqref{eq: weak} fulfills $w\in H^3(\Om)\cap H^2_0(\Om)$ and there exists a constant $C$ such that 
\begin{equation}\label{eq: H3 regularity}
\|w\|_{H^{3}(\Om)}\le C\|f\|_{H^{-1}(\Om)}.
\end{equation}
\end{proposition}
The above result yields, that the operator $\mathcal A\colon H^2_0(\Om) \to H^{-2}(\Om)$, defined by 
    \begin{equation}\label{eq:def_biharmonic_operator}
        \Opair{\mathcal A w,\varphi} = \Oprod{\Delta w, \Delta \varphi} \qquad \textforall \varphi \in H^2_0(\Om)
    \end{equation}
    is an isomorphism.
We immediately obtain the following corollary.
\begin{corollary}\label{cor:biharmonic}
Let $g \in H^1(\Omega)$ and $w \in H^2_0(\Omega)$ be the weak solution to the problem 
\begin{equation}\label{eq:Fourth weak}
	\Oprod{\Delta w, \Delta \varphi} = \Oprod{\na g,\na \varphi} \quad \textforall \varphi\in H^2_0(\Om).
	\end{equation}
	Then there exists a constant $C$  such that
	\[
	\|w\|_{H^2_0(\Omega)} \le C \|\na g\|_{L^2(\Om)}.
	\]
  If $\Omega$ is in addition convex, then $w \in H^3(\Omega)\cap H^2_0(\Omega)$ and
\begin{equation}\label{eq: H3 regularity 2}
\|w\|_{H^{3}(\Om)}\le C\|\na g\|_{L^2(\Om)}.
\end{equation}	
\end{corollary}
If $\Om$ is polygonal, but not necessarily convex, the results of 
  \cite{bacuta_shift_2002,blum_boundary_1980}
     yield, that there exists 
    some $\alpha_\Om > 0$, such that for any $0 \le \alpha < \alpha_\Om$, the solution $w$ of \eqref{eq: weak}
    for the data $f \in H^{-2 + \alpha}(\Om)$ satisfies a bound
    \begin{equation*}
        \|w\|_{H^{2 + \alpha}(\Om)} \le C \norm{f}_{H^{-2+\alpha}(\Om)}.
    \end{equation*}
We consider also an associated eigenvalue problem,  the so called buckling eigenvalue problem
\begin{equation}\label{eq:egenv}
\begin{aligned}
	\Delta^2 w &= -\lambda \Delta w && \text{in } \Omega,\\
	w = \partial_n w &=0 && \text{on } \partial \Omega,
\end{aligned}
\end{equation}
see, e.g., \cite{Wieners:1996,Antunes:2011}.

\begin{lemma}\label{lemma:eigen}
There exists a sequence of eigenvalues of \eqref{eq:egenv}
\[
0 < \lambda_1 \le \lambda_2 \le \dots, \quad \text{with } \lambda_k \to \infty, \;\text{for } k \to \infty
\]
and the corresponding eigenfunctions $w_k \in H^2_0(\Omega)$, which are orthogonal with respect to both $H^2_0(\Omega)$ and $H^1_0(\Omega)$ inner products, i.e.
\[
\Oprod{\Delta w_k,\Delta w_j} = \Oprod{\nabla w_k,\nabla w_j} = 0 \quad \text{for } k \neq j
\]
and form a basis for both $H^2_0(\Omega)$ and $H^1_0(\Omega)$.
\end{lemma}
\begin{proof}
We consider a bounded linear operator $T \colon H^1_0(\Omega) \to H^1_0(\Omega)$ defined as follows:
\begin{equation}\label{opeartor:T}
	T(g)\in H^2_0(\Omega) \quad:\quad \Oprod{\Delta T(g), \Delta \phi} 
  = \Oprod{\nabla g,\nabla \phi} \quad \text{for all } \phi \in H^2_0(\Omega).
\end{equation}
Note, that by \Cref{cor:biharmonic} for every $g \in H^1_0(\Omega)$, $T(g)\in H^2_0(\Omega)$ is well-defined. Since $H^2_0(\Omega)$ is compactly embedded into $H^1_0(\Omega)$ the operator $T \colon H^1_0(\Omega) \to H^1_0(\Omega)$ is compact. It is straightforward to check that $T$ is symmetric with respect to the $H^1_0(\Omega)$ inner product, i.e. 
\[
\Oprod{\nabla T(g), \nabla \phi} = \Oprod{\nabla g, \nabla T(\phi)} 
\quad \text{for all } g,\phi \in H^1_0(\Omega),
\]
and positive definite, since
\[
\Oprod{\nabla T(g), \nabla g} = \|\Delta T(g)\|^2_{L^2(\Omega)} \ge 0
\]
and $\Oprod{\nabla T(g), \nabla g} =0$ implies $T(g)=0$ and thus $g=0$. Thus, $T$ is also injective. We apply the
spectral theorem for compact normal operators, see, e.g., \cite[Theorem 12.12]{Alt:2016}, and obtain the result of the lemma.
\end{proof}

\begin{remark}
The image of the operator $T \colon H^1_0(\Omega) \to H^1_0(\Omega)$ defined in \eqref{opeartor:T} is given as
\[
\operatorname{Im}(T) = \Set{w \in H^2_0(\Omega) | \text{there exists } g \in H^1_0(\Omega) \text{ fulfilling~\eqref{eq:Fourth weak}}}.
\]
By \Cref{cor:biharmonic} there holds $\operatorname{Im}(T) = H^3(\Omega)\cap H^2_0(\Omega)$ if $\Omega$ is convex. Since the operator $T$ is injective we can define the inverse
\[
T^{-1} \colon \operatorname{Im}(T) \to H^1_0(\Omega),
\]
which fulfills
\[
\Oprod{\nabla T^{-1}(w),\nabla \phi} = \Oprod{\Delta w, \Delta \phi} 
\quad \text{for all } \phi \in H^2_0(\Omega).
\]
In the numerical analysis below, we use an operator $A_h$ \eqref{eq:Ah}, which is a discrete analog of  $T^{-1}$.
\end{remark}

\subsection{Results for time-dependent problems}
We call $\psi$ a weak solution of \eqref{eq:transient:Fourth} if 
\begin{equation}\label{eq:XXdefinition}
\psi \in \XX := \Set{v \in L^2(I;H^2_0(\Omega)) | \partial_t \Delta v \in L^2(I;H^{-2}(\Omega))}
\end{equation}
fulfills
\begin{equation}\label{eq:transient_biharmonic_weak}
    \begin{aligned}
	-\langle\partial_t \Delta \psi, \varphi\rangle+(\Delta \psi, \Delta \varphi)_{\IOm} &= \langle f, \varphi \rangle \quad &\text{for all }\; \varphi\in \YY := L^2(I;H^2_0(\Omega)), \\
	\psi(0) &= \psi_0. &
    \end{aligned}
\end{equation}
Here, $\langle \cdot, \cdot \rangle$ denotes the duality product between $L^2(I;H^{-2}(\Omega))$ 
and $L^2(I;H^2_0(\Omega))$.

Before we prove the existence and regularity results for \eqref{eq:transient_biharmonic_weak}, we 
    show two technical results regarding continuity in time, if some Hilbert space regularities are known.
\begin{lemma}\label{lemm:H1_continuity}
    Let $\psi \in L^2(I;H^2_0(\Om))$ satisfy $\partial_t \Delta \psi \in L^2(I;H^{-2}(\Om))$.
    Then (after a possible redefinition on a zero measure subset of $I$)
    $\psi \in C(\bar I;H^1_0(\Om))$ and there holds a bound
    \begin{equation*}
        \norm{\psi}_{C(\bar I;H^1_0(\Om))} 
        \le C \left( \norm{\psi}_{L^2(I;H^2_0(\Om))} 
        + \norm{\partial_t \Delta \psi}_{L^2(I;H^{-2}(\Om))}\right).
    \end{equation*}
    Thus the space $\XX$ defined in \eqref{eq:XXdefinition} satisfies the continuous embedding:
    $\XX \hookrightarrow C(\bar I;H^1_0(\Om))$.
\end{lemma}
\begin{proof}
    We adapt the proof of \cite[Section 5.9, Theorem 3]{Evans:2010}, and for any $\varepsilon > 0$ 
    introduce a mollifier $\eta_\varepsilon$ \revB{in time},
    which leads to the mollified functions $\psi_\varepsilon = \eta_\varepsilon * \psi$.
    \revB{Note that $\psi_\varepsilon \to \psi$ in $L^2(I;H^2_0(\Om))$,
        $\partial_t \Delta \psi_\varepsilon \to \partial_t \Delta \psi$ in $L^2(I;H^{-2}(\Om))$ and 
        $\psi_\varepsilon(t) \to \psi(t)$ in $H^2_0(\Om)$ for a.e. 
        $t \in I$, see \cite[Appendix C, Theorem 7]{Evans:2010}.
    }
    For $\varepsilon,\delta > 0$ it holds due to the smoothness in time of $\psi_\varepsilon,\psi_\delta$ and 
    $\psi(s) \in H^2_0(\Om)$ for a.e. $s \in I$
    \begin{align*}
        \partial_t \left(\norm{\nabla \left(\psi_\varepsilon(\tau) - \psi_\delta(\tau)\right)}_{L^2(\Om)}^2\right)
        &= 2\Oprod{\partial_t \nabla (\psi_\varepsilon(\tau) - \psi_\delta(\tau)),\nabla(\psi_\varepsilon(\tau) - \psi_\delta(\tau))}\\
        &= - 2\Oprod{\partial_t \Delta (\psi_\varepsilon(\tau) - \psi_\delta(\tau)),(\psi_\varepsilon(\tau) - \psi_\delta(\tau))}
    \end{align*}
    for all \revB{$\tau \in I$. By the fundamental theorem of calculus, integrating the above 
    identity over any arbitrary subinterval $(s,t)\subset I$,}
    it holds for all $s,t \in I$
    \begin{equation*}
        \norm{\nabla \left(\psi_\varepsilon(t) - \psi_\delta(t)\right)}_{L^2(\Om)}^2
        =
        \norm{\nabla \left(\psi_\varepsilon(s) - \psi_\delta(s)\right)}_{L^2(\Om)}^2
        - \int_s^t 2\Oprod{\partial_t \Delta (\psi_\varepsilon(\tau) - \psi_\delta(\tau)),(\psi_\varepsilon(\tau) - \psi_\delta(\tau))} \, d \tau.
    \end{equation*}
    \revB{Applying Cauchy-Schwarz inequality to the time integral and extending the domain of 
        integration to the whole $I$ gives
    \begin{equation*}
        \norm{\nabla \left(\psi_\varepsilon(t) - \psi_\delta(t)\right)}_{L^2(\Om)}^2
        \le
        \norm{\nabla \left(\psi_\varepsilon(s) - \psi_\delta(s)\right)}_{L^2(\Om)}^2
        + 2\norm{\partial_t \Delta (\psi_\varepsilon - \psi_\delta)}_{L^2(I;H^{-2}(\Om))}
            \norm{\psi_\varepsilon - \psi_\delta}_{L^2(I;H^2_0(\Om))}.
    \end{equation*}
    Due to $\psi_\varepsilon \to \psi$ a.e. in $I$, we can now} fix some $s \in I$ for 
    which $\psi_\varepsilon(s) \to \psi(s)$ in $H^1_0(\Om)$. 
    \revB{As $\varepsilon, \delta \to 0$, all terms on the right-hand side converge to $0$.
        Moreover, since the right-hand side is independent of $t \in I$, we have thus shown
    \begin{equation*}
        \limsup_{\varepsilon,\delta \to 0} \, \sup_{t \in I} \, 
        \norm{\nabla \left(\psi_\varepsilon(t) - \psi_\delta(t)\right)}_{L^2(\Om)}^2 = 0,
    \end{equation*}
    i.e., $\nabla \psi_\varepsilon$ is a Cauchy sequence in $C(\bar I;L^2(\Om)^2)$.
    Hence, it has a limit in $C(\bar I;L^2(\Om)^2)$, and since $\psi_\varepsilon(t) \to \psi(t)$ in $H^1_0(\Om)$
    for a.e. $t \in I$, a redefinition of $\psi$ on a set of measure zero gives
    }
%     As a consequence, it holds
%     \begin{equation*}
%         \limsup_{\varepsilon,\delta \to 0} \, \sup_{t \in I} 
%         \norm{\nabla \left(\psi_\varepsilon(t) - \psi_\delta(t)\right)}_{L^2(\Om)}^2
%         \le 2\lim_{\varepsilon,\delta \to 0} \norm{\partial_t \Delta (\psi_\varepsilon - \psi_\delta)}_{L^2(I;H^{-2}(\Om))}\norm{\psi_\varepsilon - \psi_\delta}_{L^2(I;H^2_0(\Om))}=0.
%     \end{equation*}
    \begin{equation}\label{eq:time_continuity_nabla}
        \nabla \psi \in C(\bar I;L^2(\Om)^2).
    \end{equation}
    Now that we have fixed a representation of $\psi$ with $\nabla \psi$ being continuous in time, we have 
    to take special care with further redefinitions on zero sets.
    %From now on this representant of $\psi$ will be fixed and no redifinitions on zero sets are allowed.
    To conclude the proof, note that $\psi \in L^2(I;H^2_0(\Om))$ especially implies that 
    $\psi(t) \in H^1_0(\Om)$ for a.e. $t \in I$. Hence, for every $t^* \in \bar I$, there exists a sequence 
    $\{t_n\} \subset I$ with $t_n \to t^*$, $\psi(t_n) \in H^1_0(\Om)$ for all $n$ and by 
    \eqref{eq:time_continuity_nabla}
    \begin{equation}\label{eq:nabla_convergence}
     \norm{\nabla \psi(t_n) - \nabla \psi(t^*)}_{L^2(\Om)} \to 0, \ n\to \infty.
    \end{equation}
    Especially $\{\psi(t_n)\}$ is a Cauchy sequence in $H^1_0(\Om)$ and thus 
    $\psi(t_n) \to \psi^* \in H^1_0(\Om)$ for some $\psi^* \in H^1_0(\Om)$.
    From \eqref{eq:nabla_convergence}, we deduce $\nabla \psi(t^*) = \nabla \psi^*$, and especially
    that $\psi^*$ is independent of the explicit choice of the sequence $\{t_n\}$.
    Thus, we can redefine $\psi(t^*) := \psi^*$, without breaking the property \eqref{eq:time_continuity_nabla}.
    %Doing this for all $t^*$ on the set of measure zero, where $\psi(t^*)\not \in H^1_0(\Om)$, 
   % finally yields $\psi \in C(\bar I;H^1_0(\Om))$.
    \revB{Let us lastly show the proposed estimate. By applying the same arguments as above 
        to $\psi_\varepsilon$ alone, instead of the difference $\psi_\varepsilon - \psi_\delta$, we obtain
        for any $t,s \in I$
    \begin{equation*}
        \norm{\nabla \psi_\varepsilon(t)}_{L^2(\Om)}^2
        = \norm{\nabla \psi_\varepsilon(s)}_{L^2(\Om)}^2
        - \int_s^t 2\Oprod{\partial_t \Delta \psi_\varepsilon(\tau),\psi_\varepsilon(\tau)} \, d \tau.
    \end{equation*}
    Since in the above construction $\psi$ is the pointwise limit of $\psi_\varepsilon$, this implies also
    \begin{align*}
        \norm{\nabla \psi(t)}_{L^2(\Om)}^2
        & = \norm{\nabla \psi(s)}_{L^2(\Om)}^2
        - \int_s^t 2\Opair{\partial_t \Delta \psi(\tau),\psi(\tau)} \, d \tau\\
        & \le 
         \norm{\nabla \psi(s)}_{L^2(\Om)}^2
        + 2\norm{\partial_t \Delta \psi}_{L^2(I;H^{-2}(\Om))} \norm{\psi}_{L^2(I;H^2_0(\Om))}.
    \end{align*}
    Instead of fixing $s$, we now may integrate with respect to $s$, which yields 
    \begin{equation*}
        T \norm{\nabla \psi(t)}_{L^2(\Om)}^2
        = \int_0^T \norm{\nabla \psi(s)}_{L^2(\Om)}^2 \, d s
        + 2 T \norm{\partial_t \Delta \psi}_{L^2(I;H^{-2}(\Om))} \norm{\psi}_{L^2(I;H^2_0(\Om))}.
    \end{equation*}
    Dividing by $T$, applying the estimate $\norm{\nabla v}_{L^2(\Om)} \le \norm{v}_{H^2_0(\Om)}$ for any
    $v \in H^2_0(\Om)$ to the first term on the right-hand side,
    and Youngs inequality to the second term on the right-hand side gives
    \begin{equation*}
        \norm{\nabla \psi(t)}_{L^2(\Om)}^2 
        \le C\left( \norm{\partial_t \Delta \psi}_{L^2(I;H^{-2}(\Om))}^2 + \norm{\psi}_{L^2(I;H^2_0(\Om))}^2\right).
    \end{equation*}
    Taking the square root concludes the proof.
    }
\end{proof}
\begin{lemma}\label{lemm:H20_continuity}
    Let $\psi \in L^2(I;H^2_0(\Om))$ satisfy $\partial_t \psi \in L^2(I;H^{1}_0(\Om))$ 
    and $\mathcal A \psi \in L^2(I;H^{-1}(\Om))$.
    Then (after a possible redefinition on a measure zero subset of $I$)
    $\psi \in C(\bar I;H^2_0(\Om))$ and there holds a bound
    \begin{equation*}
        \norm{\psi}_{C(\bar I;H^2_0(\Om))} 
        \le C \left( \norm{\mathcal A \psi}_{L^2(I;H^{-1}(\Om))} 
        + \norm{\partial_t \psi}_{L^2(I;H^1_0(\Om))}\right).
    \end{equation*}
\end{lemma}
\begin{proof}
  Similarly to the proof of \Cref{lemm:H1_continuity}, we consider the mollified functions 
  $\psi_\varepsilon = \psi*\eta_\varepsilon$. As 
  $\partial_t(\psi*\eta_\varepsilon) = \psi *(\partial_t \eta_\varepsilon)$,
  and $\psi(t) \in H^2_0(\Om)$ for 
  a.e. $t \in I$, it holds $\partial_t(\psi*\eta_\varepsilon) \in H^2_0(\Om)$ for all $t \in I$, and thus
  we obtain via integration by parts:
  \begin{align*}
        \partial_t \left(\norm{\Delta \left(\psi_\varepsilon(t) - \psi_\delta(t)\right)}_{L^2(\Om)}^2\right)
        &= 2\Oprod{\partial_t \Delta (\psi_\varepsilon(t) - \psi_\delta(t)),\Delta(\psi_\varepsilon(t) - \psi_\delta(t))}\\
        &= 2\revB{\Opair{\partial_t (\psi_\varepsilon(t) - \psi_\delta(t)), \mathcal A(\psi_\varepsilon(t) - \psi_\delta(t))}}.
  \end{align*}
   By the fundamental theorem of calculus, it holds for all $s,t \in I$
    \begin{equation*}
        \norm{\Delta \left(\psi_\varepsilon(t) - \psi_\delta(t)\right)}_{L^2(\Om)}^2
        =
        \norm{\Delta \left(\psi_\varepsilon(s) - \psi_\delta(s)\right)}_{L^2(\Om)}^2
        + \int_s^t 2\revB{\Opair{\partial_t (\psi_\varepsilon(\tau) - \psi_\delta(\tau)),
            \mathcal A(\psi_\varepsilon(\tau) - \psi_\delta(\tau))}} \, d \tau.
    \end{equation*}
    Hence, fixing any $s \in I$ with $\Delta \psi_\varepsilon (s) \to \Delta \psi(s)$ in $L^2(\Om)$:
    \begin{equation*}
        \limsup_{\varepsilon,\delta \to 0} \, \sup_{t \in I} \,
        \norm{\Delta \left(\psi_\varepsilon(t) - \psi_\delta(t)\right)}_{L^2(\Om)}^2
        \le \lim_{\varepsilon,\delta \to 0} \norm{\partial_t  (\psi_\varepsilon - \psi_\delta)}_{L^2(I;H^1_0(\Om))}\norm{\mathcal A(\psi_\varepsilon - \psi_\delta)}_{L^2(I;H^{-1}(\Om))}=0.
    \end{equation*}
    Thus, we obtain as before, that 
    \begin{equation*}
      \Delta \psi \in C(\bar I;L^2(\Om)).
    \end{equation*}
    As in the proof of \Cref{lemm:H1_continuity}, using this result, and the fact that 
    $H^2_0(\Om) \subset H^2(\Om)$ is closed, we deduce that in fact $\psi \in C(\bar I;H^2_0(\Om))$.
\end{proof}

For \eqref{eq:transient_biharmonic_weak} we have the following existence and stability result. 
\begin{theorem}[Existence, uniqueness and stability]\label{thm:existence_uniqueness_standard_regularity}
Let $f \in L^2(I;H^{-2}(\Om))$ and $\psi_0 \in H^1_0(\Om)$. Then there exists a unique solution $\psi\in \XX$ to \eqref{eq:transient_biharmonic_weak}. Moreover, $\psi\in C(\bar{I}; H^1_0(\Om))$ and the following estimate holds
$$
\|\partial_t \Delta \psi\|_{L^2(I;H^{-2}(\Omega))} 
+ \|\psi\|_{L^2(I;H^2_0(\Omega))}
+\|\psi\|_{C(\bar{I}; H^1_0(\Om))}\le C\left(\|f\|_{L^2(I;H^{-2}(\Om))}+\|\na \psi_0\|_{L^2(\Om)}\right).
$$
\end{theorem}
\begin{proof}
The proof uses a Galerkin procedure similar to the proof of \cite[Section 7.1.2, Theorem 3]{Evans:2010} in the case of the heat equation. We use the eigenfunctions $w_k \in H^2_0(\Omega)$ from \Cref{lemma:eigen} and normalize them such that $\|\nabla w_k\|_{L^2(\Omega)} = 1$. For $n \in \N$ we consider the space
\[
V_n = \operatorname{span}\set{w_k | k=1,2,\dots, n}
\] 
and a Galerkin approximation to \eqref{eq:transient_biharmonic_weak} given as $\psi_n \colon \bar I \to V_n$ and fulfilling
\begin{equation}\label{eq:galerkin_approximation}
\begin{aligned}
    \Oprod{\nabla \partial_t \psi_n(t), \nabla v_n} + \Oprod{\Delta \psi_n(t),\Delta v_n} 
    &= \langle f(t), v_n \rangle_{H^{-2}(\Omega),H^2_0(\Omega)} 
    \quad &&\text{for all }\; v_n \in V_n \text{ and a.a. } t \in I\\
\Oprod{\nabla \psi_n(0),\nabla v_n} &= \Oprod{\nabla \psi_0,\nabla v_n} \quad &&\text{for all }\; v_n \in V_n.
\end{aligned}
\end{equation}
The existence and uniqueness of a solution of this system of ODEs follows directly.
    The initial condition in \eqref{eq:galerkin_approximation}, can compactly be written as 
    \begin{equation}\label{eq:galerkin_approximation_initial_condition}
        \psi_n(0) = P_n \psi_0,
    \end{equation}
    where $P_n \colon H^1_0(\Omega) \to V_n$ is the orthogonal projection defined by
\begin{equation}\label{eq:definition_Pn}
\Oprod{\nabla P_n v,\nabla v_n} = \Oprod{\nabla v,\nabla v_n}
\qquad \text{for all  $v_n \in V_n$}.
\end{equation}
It is straightforward to check, that this projection is stable in $H^1_0(\Om)$, i.e., there holds 
    \begin{equation}\label{eq:stability_Pn}
        \|P_n v\|_{H^1_0(\Om)} \le \|v\|_{H^1_0(\Om)}.
    \end{equation}
Testing \eqref{eq:galerkin_approximation} with $v_n= \psi_n(t)$ we obtain for a.e. $t \in I$
$$
\frac{1}{2}\frac{d}{dt}\|\na \psi_n(t)\|^2_{L^2(\Om)}+\|\Delta \psi_n(t)\|^2_{L^2(\Om)} = \langle f(t), \psi_n(t) \rangle_{H^{-2}(\Omega),H^2_0(\Omega)} \le \|f(t)\|_{H^{-2}(\Om)} \|\psi_n(t)\|_{H^2(\Om)}.
$$
Integrating in time over $(0,t)$ for arbitrary $t \in I$,
and  using  the fact that due to \Cref{lemm:h2_norm_equivalence} 
$\|\psi_n(t)\|_{H^2_0(\Om)} =\|\Delta \psi_n(t)\|_{L^2(\Om)}$ we obtain
after an application of the Young's inequality
$$
\frac{1}{2} \|\na \psi_n(t)\|^2_{L^2(\Om)}+ \|\psi_n\|^2_{L^2((0,t);H^2_0(\Omega))}
\le \frac{1}{2}\|\na \psi_n(0)\|^2_{L^2(\Om)}
+\frac{1}{2} \|\psi_n\|^2_{L^2((0,t);H^2_0(\Omega))} +\frac{1}{2}\|f\|^2_{L^2((0,t);H^{-2}(\Om))}.
$$
Using the fact that, due to \eqref{eq:stability_Pn} it holds
$\|\nabla \psi_n(0)\|^2_{L^2(\Om)} \le \|\nabla \psi_0\|^2_{L^2(\Om)}$, 
and $t \in I$ was arbitrary, we get
\begin{equation}\label{eq: continuous estimate 1}
    \|\na \psi_n\|^2_{L^\infty(I;L^2(\Om))}+ \|\psi_n\|^2_{L^2(I;H^2_0(\Omega))}
    \le \|\na \psi_0\|^2_{L^2(\Om)} +\|f\|^2_{L^2(I;H^{-2}(\Om))}.
\end{equation}
In the next step we show the boundedness of $\|\partial_t \Delta \psi_n\|_{L^2(I;H^{-2}(\Omega))}$.
By the orthonormality in $H^1_0(\Om)$ of the basis we get for every $v \in H^1_0(\Omega)$
\[
v = \sum_{k=1}^\infty \Oprod{\nabla v,\nabla w_k} w_k \quad \text{and} \quad P_n v 
= \sum_{k=1}^n \Oprod{\nabla v,\nabla w_k} w_k.
\]
Moreover, if $v \in H^2_0(\Om)$ and $v_n \in V_n$, then by orthogonality of $\{w_k\}$ in $H^2_0(\Om)$:
  \begin{equation}\label{eq:discrete_tested_with_continuous}
  \Oprod{\Delta v_n,\Delta v} 
  = \Oprod{\Delta v_n,\sum_{k=1}^{\infty}\Delta w_k \Oprod{\na v,\na w_k}}
  = \Oprod{\Delta v_n,\sum_{k=1}^{n}\Delta w_k \Oprod{\na v,\na w_k}}
  = \Oprod{\Delta v_n,\Delta P_nv}.
\end{equation}
By the same arguments, we obtain for any $v\in H^2_0(\Omega)$ 
\[
\|P_n v\|_{H^2_0(\Omega)}^2 = \|\Delta P_n v\|^2_{L^2(\Omega)} 
= \sum_{k=1}^n\Oprod{\nabla v,\nabla w_k}^2 \|\Delta w_k\|^2_{L^2(\Omega)} 
\le \sum_{k=1}^\infty\Oprod{\nabla v,\nabla w_k}^2 \|\Delta w_k\|^2_{L^2(\Omega)} 
= \|\Delta v\|^2_{L^2(\Omega)}.
\]
Let $v \in H^2_0(\Omega)$ be arbitrary. As $\psi_n(t) \in V_n$ a.e. in $I$, there holds 
$\partial_t \psi_n(t) \in V_n$ a.e. in $I$, and thus for a.e. $t \in I$, we get
\[
\begin{aligned}
(\partial_t \Delta \psi_n(t),v)_\Omega &= -(\nabla \partial_t \psi_n(t),\nabla v)_\Omega = -(\nabla \partial_t \psi_n(t),\nabla P_n v)_\Omega\\
&= (\Delta \psi_n(t),\Delta P_n v)_\Omega - \langle f(t), P_n v \rangle_{H^{-2}(\Omega),H^2_0(\Omega)}\\
&\le \|\Delta \psi_n(t)\|_{L^2(\Omega)} \|\Delta P_n v\|_{L^2(\Omega)} 
+ \|f(t)\|_{H^{-2}(\Om)} \|P_n v\|_{H^2_0(\Om)}\\
& \le \left(\|\psi_n(t)\|_{H^2_0(\Omega)} + \|f(t)\|_{H^{-2}(\Om)} \right) \|v\|_{H^2_0(\Omega)}.
\end{aligned}
\]
Thus for a.e. $t \in I$, there holds
\[
\|\partial_t \Delta \psi_n(t)\|_{H^{-2}(\Omega)} \le \|\psi_n(t)\|_{H^2_0(\Omega)} + \|f(t)\|_{H^{-2}(\Om)}.
\]
Squaring, integrating w.r.t $t$ and using \eqref{eq: continuous estimate 1}, we get
\begin{equation}\label{eq:galerkin_approx_bound_dt}
    \|\partial_t \Delta \psi_n\|_{L^2(I;H^{-2}(\Omega))} 
    \le  C\left(\|\na \psi_0\|_{L^2(\Om)} +\|f\|_{L^2(I;H^{-2}(\Om))}\right).
\end{equation}
From these estimates we obtain existence of a subsequence of $\psi_n$ (denoted again by $\psi_n$) and some $\psi \in \XX$ such that
\begin{equation}\label{eq:weak_convergence_galerkin_approx}
\psi_n \rightharpoonup \psi \text{ in } L^2(I;H^2_0(\Omega)) \;\text{ and }\; \partial_t \Delta \psi_n \rightharpoonup \partial_t \Delta \psi \text{ in } L^2(I;H^{-2}(\Omega)).
\end{equation}
To conclude the proof, let $\varphi^N \in C^1(\bar I;H^2_0(\Om))$ be of the form 
$\varphi^N = \sum_{j=1}^N \zeta_j(t)w_j$ where $N \in \N$ is fixed and $\zeta_i(t)\in C^1(\bar{I})$.
Then, testing the first equation of \eqref{eq:galerkin_approximation} for $n > N$
with $\varphi^N(t)$ and integrating over $I$ yields
\[
    -\IOpair{\partial_t \Delta \psi_n, \varphi^N}+\IOprod{\Delta \psi_n, \Delta \varphi^N}
	%= -\langle\partial_t \Delta \psi_n, P_n\varphi\rangle+(\Delta \psi_n, \Delta P_n \varphi)_{\IOm} 
        = \langle f, \varphi^N \rangle.
\]
Using the weak convergences \eqref{eq:weak_convergence_galerkin_approx} thus yields
\[
    -\IOpair{\partial_t \Delta \psi, \varphi^N}+\IOprod{\Delta \psi, \Delta \varphi^N}
        = \langle f, \varphi^N \rangle.
\]
As functions of the form of $\varphi^N$ are dense in $\YY=L^2(I;H^2_0(\Omega))$, we have just shown the first equation of 
\eqref{eq:transient_biharmonic_weak}.
To show, that the initial condition is   satisfied, let $\varphi \in C^1(\bar I;H^2_0(\Om))$ be arbitrary
with $\varphi(T) = 0$.
Testing the first equation of \eqref{eq:transient_biharmonic_weak} with $\varphi$, yields after integration
by parts in time and space, that 
\[
    \revB{-\IOprod{\nabla \psi, \partial_t \nabla \varphi}}+\IOprod{\Delta \psi, \Delta \varphi}
    = \langle f, \varphi \rangle + \Oprod{\nabla \psi(0), \nabla \varphi(0)}.
\]
Moreover, using integration by parts in time, $\varphi(T) = 0$,
the definition of $P_n$ \eqref{eq:definition_Pn}, as well as
\eqref{eq:discrete_tested_with_continuous} applied to every $t \in I$, we have
\begin{align*}
    \revB{-\IOprod{\nabla \psi_n, \partial_t \nabla \varphi}}+\IOprod{\Delta \psi_n, \Delta \varphi}
    &=\revB{\IOprod{\partial_t \nabla \psi_n, \nabla \varphi}}+\IOprod{\Delta \psi_n, \Delta \varphi}
    + \Oprod{\nabla \psi_n(0), \nabla \varphi(0)}\\
    &=\revB{\IOprod{\partial_t \nabla \psi_n, \nabla P_n \varphi}}+\IOprod{\Delta \psi_n, \Delta P_n \varphi}
    + \Oprod{\nabla \psi_n(0), \nabla \varphi(0)}\\
    %&= -\IOpair{\nabla \psi_n, \partial_t \nabla P_n\varphi}+\IOprod{\Delta \psi_n, \Delta P_n\varphi}\\
    &= \IOpair{ f, P_n\varphi } + \Oprod{\nabla \psi_n(0), \nabla \varphi(0)}\\
    &= \IOpair{ f, P_n\varphi } + \Oprod{\nabla P_n\psi_0, \nabla \varphi(0)}.
\end{align*}
Hence, subtracting the two equations, and taking the limit yields due to the weak convergence 
\eqref{eq:weak_convergence_galerkin_approx} and \eqref{eq:stability_Pn}
%\begin{align*}
\[
    0= \lim_{n\to\infty} \left[ \IOpair{ f, \varphi - P_n \varphi} +
     \Oprod{\nabla(\psi(0) - P_n\psi_0),\nabla\varphi(0)}\right].
     \]
%\end{align*}
Using the convergence $P_n \varphi(t) \to \varphi(t)$ in $H^2_0(\Om)$ for all $t \in I$, together
with dominated convergence, as well as $P_n \psi_0 \to \psi_0$ in $H^1_0(\Om)$ yields
\begin{equation*}
    \Oprod{\nabla(\psi(0) - \psi_0),\nabla \varphi(0)} = 0.
\end{equation*}
As $\varphi(0)$ can be arbitrary, we have concluded with the proof of $\psi(0)=\psi_0$.
From the estimates \eqref{eq: continuous estimate 1},\eqref{eq:galerkin_approx_bound_dt}, we know that 
$\psi$ satisfies the bounds 
$$
\|\partial_t \Delta \psi\|_{L^2(I;H^{-2}(\Omega))} 
+ \|\psi\|_{L^2(I;H^2_0(\Omega))}
\le C\left(\|f\|_{L^2(I;H^{-2}(\Om))}+\|\na \psi_0\|_{L^2(\Om)}\right).
$$
The continuity in time is then a direct consequence of \Cref{lemm:H1_continuity}.
Choosing $\psi_0 =0$ and $f=0$ directly implies uniqueness.
\end{proof}

  \begin{remark}
    The stability estimate of \Cref{thm:existence_uniqueness_standard_regularity} bounds the term
    $\|\pa_t \Delta \psi\|_{L^2(I;H^{-2}(\Om))}$ by norms of the data. One could expect that this results in $\pa_t \psi \in L^2(I \times \Om)$, which is, however, not the case as we show in \Cref{sec:appendix} by means of a counterexample. In \cite[Theorem 2.3]{KimQuasiGeostrophic2025} a regularity result $\pa_t \psi \in L^2(I;H^{-1}(\Omega))$ is claimed for a version of \eqref{eq:transient_biharmonic_weak} with some non-linear terms and the same assumptions on the data as in Theorem \ref{thm:existence_uniqueness_standard_regularity}. In \Cref{sec:appendix} we show that even this result is in general not true for \eqref{eq:transient_biharmonic_weak}. 
    %From this it is however not clear, whether $\pa_t \psi \in L^2(I \times \Om)$ holds. 
    %In fact such a result seems doubtful, as by a duality argument this would relate to
    %$\Delta^{-1}\colon L^2(\Om) \to H^2_0(\Om)$ which does not hold in general even on convex domains.
  \end{remark}

Under some additional conditions on the data, we can show higher-regularity result.
\begin{theorem}\label{prop:weak_stability}
    Let $f \in L^2(I;H^{-1}(\Omega))$ and $\psi_0 \in H^2_0(\Omega)$.
    Then the solution to \eqref{eq:transient_biharmonic_weak} additionally has the following regularities:
    $\partial_t \psi\in L^2(I;H^1_0(\Om))$, $\mathcal A \psi \in L^2(I;H^{-1}(\Om))$,
    and $\psi \in C(\bar I;H^2_0(\Omega))$. The weak form can be written as 
\begin{equation}
    \begin{aligned}
	\IOprod{\partial_t \na \psi, \na \varphi}
        +\IOprod{\Delta \psi, \Delta \varphi}
        &= \langle f, \varphi \rangle \quad &\text{for all }\; \varphi\in \YY, \\
	\psi(0) &= \psi_0, &
    \end{aligned}
\end{equation}
        and there holds
	\[
	\|\na\partial_t \psi\|_{L^2(I;L^{2}(\Om))} 
  + \|\mathcal A \psi\|_{L^2(I;H^{-1}(\Om))} 
  + \|\psi\|_{C(\bar I;H^2_0(\Om))}
  \le C \left( \|f\|_{L^2(I;H^{-1}(\Omega))} + \|\psi_0\|_{H^2_0(\Omega)}\right).
	\]
	If $\Omega$ is in addition convex, then $\psi \in L^2(I;H^3(\Om)\cap H^2_0(\Om))$ and there holds
 	\[
 \|\psi\|_{L^2(I;H^3(\Om))} 
 \le C \left( \|f\|_{L^2(I;H^{-1}(\Omega))} + \|\psi_0\|_{H^2_0(\Omega)}\right).
 \]
  \end{theorem}
\begin{proof}
We use the Galerkin approximation from the proof of \Cref{thm:existence_uniqueness_standard_regularity} and
test \eqref{eq:galerkin_approximation} with $v_n = \partial_t \psi_n(t) \in V_n$. This results for 
a.e. $t \in I$ in
$$
\|\na \partial_t \psi_n(t)\|^2_{L^2(\Om)}+\frac{1}{2}\frac{d}{dt}\|\Delta \psi_n(t)\|^2_{L^2(\Om)} = \langle f(t),\partial_t \psi_n(t)\rangle_{H^{-1}(\Omega),H^1_0(\Omega)}\le \|f(t)\|_{H^{-1}(\Om)}\|\nabla \partial_t \psi_n(t)\|_{L^2(\Om)}.
$$
Integrating over $I$ and using the Young's inequality, we have
\begin{equation}\label{eq:nabla_dt_psi_bound}
\|\na \partial_t \psi_n\|^2_{I\times\Om}
+ \frac{1}{2} \|\Delta \psi_n(T)\|^2_{L^2(\Om)}
\le \frac{1}{2} \|\Delta \psi_n(0)\|^2_{L^2(\Om)}
+\frac{1}{2}\|\nabla \partial_t \psi_n\|^2_{I\times\Om}
+\frac{1}{2}\|f\|^2_{L^2(I; H^{-1}(\Om))}.
\end{equation}
    As $\psi_0 \in H^2_0(\Om)$, due to \eqref{eq:galerkin_approximation_initial_condition} 
    and \eqref{eq:discrete_tested_with_continuous}, there holds 
    \begin{equation*}
        \|\Delta \psi_n(0)\|^2_{L^2(\Om)}
        = \|\Delta P_n \psi_0\|^2_{L^2(\Om)}
        \le \|\Delta \psi_0\|^2_{L^2(\Om)}
        = \|\psi_0\|^2_{H^2_0(\Om)}.
    \end{equation*}
    Hence, after canceling terms in \eqref{eq:nabla_dt_psi_bound}, multiplying by 2 and taking the root,
    we obtain
\begin{equation}\label{eq: continuous estimate 2}
\|\na \partial_t \psi_n\|_{I\times\Om}\le \|\psi_0\|_{H^2_0(\Om)}+\|f\|_{L^2(I; H^{-1}(\Om))}.
\end{equation}
Moreover, using integration by parts, \eqref{eq:discrete_tested_with_continuous} and 
the definition of the Galerkin approximation \eqref{eq:galerkin_approximation},
there holds for any $v \in H^2_0(\Om)$ and a.e. $t \in I$
\begin{align*}
  \Oprod{\mathcal A \psi_n(t) , v} 
  &= \Oprod{\Delta \psi_n(t) , \Delta v} 
    %= \Oprod{\Delta \psi_n(t), \sum_{j=1}^\infty \Oprod{\na v,\na w_j} \Delta w_j }
    %= \Oprod{\Delta \psi_n(t), \sum_{j=1}^n \Oprod{\na v,\na w_j} \Delta w_j }\\
  = \Oprod{\Delta \psi_n(t), \Delta P_n v}
  = -\Oprod{\partial_t \na \psi_n(t), \na P_n v}+\langle f(t), P_n v \rangle.
%  & \le C \left(\norm{\pa_t \na \psi_n(t)}_{L^2(\Om)} + \norm{f}_{H^{-1}(\Om)}\right)\norm{P_n v}_{H^1_0(\Om)}
%   \le C \left(\norm{\pa_t \na \psi_n(t)}_{L^2(\Om)} + \norm{f}_{H^{-1}(\Om)}\right)\norm{v}_{H^1_0(\Om)}.
\end{align*}
By density of $H^2_0(\Om)$ in $H^1_0(\Om)$, for all $v \in H^1_0(\Om)$ we thus have
\begin{equation*}
  \Oprod{\mathcal A \psi_n(t) , v} 
  = -\Oprod{\partial_t \na \psi_n(t), \na P_n v}+\langle f(t), P_n v \rangle.
\end{equation*}
Hence, using the stability of $P_n$, \eqref{eq:stability_Pn}, we obtain the bound
\begin{align*}
  \Oprod{\mathcal A \psi_n(t) , v} 
   &\le C \left(\norm{\pa_t \na \psi_n(t)}_{L^2(\Om)} + \norm{f(t)}_{H^{-1}(\Om)}\right)\norm{P_n v}_{H^1_0(\Om)}\\
   &\le C \left(\norm{\pa_t \na \psi_n(t)}_{L^2(\Om)} + \norm{f(t)}_{H^{-1}(\Om)}\right)\norm{v}_{H^1_0(\Om)},
\end{align*}
which leads to the estimate
$$
\norm{\mathcal A \psi_n(t)}_{H^{-1}(\Om)} 
\le C \left(\norm{\pa_t \na \psi_n(t)}_{L^2(\Om)} + \norm{f(t)}_{H^{-1}(\Om)}\right),
$$
and thus after squaring and integrating in time, and using \eqref{eq: continuous estimate 2}
\begin{equation}\label{eq:D2H-1estimate}
\norm{\mathcal A \psi_n}_{L^2(I;H^{-1}(\Om))} \le C\left(\|\psi_0\|_{H^2(\Om)}+\|f\|_{L^2(I; H^{-1}(\Om))}\right).
\end{equation}
 As in the proof of \Cref{thm:existence_uniqueness_standard_regularity}, we obtain 
    $\mathcal A \psi_n \rightharpoonup \mathcal A \psi$ in $L^2(I;H^{-1}(\Om))$ and
    $\na \pa_t \psi_n \rightharpoonup \na \pa_t \psi$ in $L^2(I;L^2(\Om)^2)$
where $\psi$ is the solution to \eqref{eq:transient_biharmonic_weak}. As  balls in the norms of $L^2(I;H^{-1}(\Om))$ and $L^2(I;L^2(\Om)^2)$ and  are weakly closed, the bounds \eqref{eq: continuous estimate 2} and 
\eqref{eq:D2H-1estimate} derived for $\psi_n$ also hold for $\psi$.
Using $\pa_t \psi \in L^2(I;H^1_0(\Om))$ and $\mathcal A \psi \in L^2(I;H^{-1}(\Om))$ together with
\Cref{lemm:H20_continuity} yields $\psi \in C(\bar I;H^2_0(\Om))$.
If $\Omega$ is convex, due to $\mathcal A \psi \in L^2(I;H^{-1}(\Om))$, we may apply the
$H^3$ regularity from \Cref{prop: elliptic} for a.e. $t \in I$.
This yields $\psi \in L^2(I;H^3(\Om)\cap H^2_0(\Om))$ together with the corresponding estimate.
\end{proof}

\begin{remark}
    In \Cref{thm:existence_uniqueness_standard_regularity,prop:weak_stability}, we have stated the stability
    results
    in terms of the data for the stream-function formulation. If the data $f,\psi_0$ are not directly given,
    but rather have to be deduced from the data for the corresponding Stokes problem, i.e., it holds
    $f = -\curl(\g)$ and $\psi_0$ solves \eqref{eq:initial_psi0}, we have the following 
    relations. If $\g \in L^2(I;\V^*)$ and $\uu_0 \in \H$, then 
    \begin{equation*}
        \|\partial_t \Delta \psi\|_{L^2(I;H^{-2}(\Omega))} 
        + \|\psi\|_{L^2(I;H^2_0(\Omega))}
        +\|\psi\|_{C(\bar{I}; H^1(\Om))}
        \le C\left(\|\g\|_{L^2(I;\V^*)}+\|\uu_0\|_{L^2(\Om)}\right).
    \end{equation*}
    If moreover $\g \in L^2(I;\H)$ and $\uu_0 \in \V$, then 
    \[
   	\|\na\partial_t \psi\|_{L^2(I;L^{2}(\Om))} 
     + \|\mathcal A \psi\|_{L^2(I;H^{-1}(\Om))} 
     + \|\psi\|_{C(\bar I;H^2_0(\Om))}
     \le C \left( \|\g\|_{L^2(I;\H)} + \|\uu_0\|_{\V}\right).
   	\]
    These are direct consequences of \Cref{lemm:initial_data_stability} and \eqref{eq:curl_isom}.
\end{remark}

\subsection{Interpolated Regularity}
The formulations of  \Cref{thm:existence_uniqueness_standard_regularity,prop:weak_stability}
were chosen to explicitly highlight the best regularity of $\psi$ that we require,
but obscure the derivation of a corresponding intermediate result.
Before we state the next corollary, let us note, that the stability estimate of
\Cref{thm:existence_uniqueness_standard_regularity} can equivalently be written as:
If $f \in L^2(I;H^{-2}(\Om))$ and $\psi_0 \in H^1_0(\Om)$, then
\begin{equation}\label{eq:parabolic_regularity_low}
  \|\partial_t \Delta \psi\|_{L^2(I;H^{-2}(\Om))} + \|\mathcal A \psi \|_{L^2(I;H^{-2}(\Om))}
  \le C \left( \|\psi_0\|_{H^1_0(\Om)} + \|f\|_{L^2(I;H^{-2}(\Om))}\right).
\end{equation}
Accordingly, \Cref{prop:weak_stability} states that, if $f \in L^2(I;H^{-1}(\Om))$ and
$\psi_0 \in H^2_0(\Om)$, then
\begin{equation}\label{eq:parabolic_regularity_high}
  \|\partial_t \Delta \psi\|_{L^2(I;H^{-1}(\Om))} + \|\mathcal A \psi \|_{L^2(I;H^{-1}(\Om))}
  \le C \left( \|\psi_0\|_{H^2_0(\Om)} + \|f\|_{L^2(I;H^{-1}(\Om))}\right),
\end{equation}
where in both statements we use the notation $\mathcal A$ for the biharmonic operator as introduced in
\eqref{eq:def_biharmonic_operator}. We then immediately obtain:
\begin{corollary}\label{corr:parabolic_regularity_interpolated}
Let $\theta \in [0,1]$, $f \in L^2(I;H^{-2+\theta}(\Omega))$ and
$\psi_0 \in H^{1+\theta}_0(\Om)$.
Then the solution to \eqref{eq:transient_biharmonic_weak} satisfies the estimate
\begin{equation}\label{eq:parabolic_regularity_interpolated}
  \|\partial_t \Delta \psi\|_{L^2(I;H^{-2+\theta}(\Om))} 
  + \|\mathcal A \psi \|_{L^2(I;H^{-2+\theta}(\Om))}
  \le C \left( \|\psi_0\|_{H^{1+\theta}_0(\Om)} + \|f\|_{L^2(I;H^{-2+\theta}(\Om))}\right).
\end{equation}
  \end{corollary}
  \begin{proof}
    The result is a direct consequence of \eqref{eq:parabolic_regularity_low} and
    \eqref{eq:parabolic_regularity_high} together with the complex interpolation.
  \end{proof}

%%%%%%%%%%%%%%%%%%%%%%%%%%%%%%%%%%%%%%%%%%%%%%%%%%%%%%%%%%%%%%%%%%%%%%%%%%%%%%%%%%%%%%%%%%%%%%%%%%%%%%%%%%%%%
\section{Fully discrete approximation}\label{sec: discretization}
In this section we consider the  discrete version of the operators presented in the previous section and introduce fully discrete Galerkin solution.  

\subsection{Spatial discretization}\label{sec: space discretization}

Let $\{\Th\}$\label{glos:triangulation} be a 
family of triangulations of $\bar \Omega$, consisting of closed simplices, where we denote by $h$ the maximum mesh-size. Let $V_h\subset H^1_0(\Omega)$ be the finite dimensional space.  
We denote by $\Pi_h\colon H^1_0(\Om) \to V_h$ the $H^1_0$-projection, defined by 
\begin{equation}\label{eq:definition_Pih}
  \Oprod{\nabla \Pi_h w,\nabla v_h} = \Oprod{\nabla w,\nabla v_h} \qquad \textforall v_h \in V_h.
\end{equation}
By definition, $\Pi_h$ is stable in $H^1_0(\Om)$, i.e., there holds 
    \begin{equation}\label{eq:stability_Pih}
        \|\Pi_h w\|_{H^1_0(\Om)} \le \|w\|_{H^1_0(\Om)}.
    \end{equation}
To allow for various space discretizations, we introduce a symmetric discrete elliptic bilinear form 
$$
a_h(\cdot,\cdot):V_h \times V_h \to \mathbb{R},
$$
where at this point we do not fix a special choice of discretization in space.
As $a_h(\cdot,\cdot)$ was defined to be symmetric and elliptic, i.e., there holds
$$
a_h(\varphi_h, \varphi_h)\geq 0\quad \textforall \varphi_h\in V_h\quad\text{and}\quad a_h(\varphi_h, \varphi_h)=0\quad \Rightarrow \quad \varphi_h=0,
$$
the bilinear form $a_h(\cdot,\cdot)$  defines a norm on $V_h$ 
\begin{equation}\label{eq:triple_norm_ah}
\vertiii{w_h}_h := \sqrt{a_h(w_h,w_h)}.
\end{equation}
In addition we also define an operator $A_h\colon V_h\to V_h$, by
\begin{equation}\label{eq:Ah}
a_h(\phi_h,\varphi_h)=(\na A_h \phi_h,\na\varphi_h)_\Om\quad \textforall\phi_h,\varphi_h\in V_h.
\end{equation}
Notice, since the bilinear form $a_h(\cdot,\cdot)$ is symmetric, the operator $A_h$ is self-adjoint 
w.r.t the $H^1_0(\Om)$ inner product, since
$$
(\na A_h \phi_h,\na\varphi_h)_\Om=a_h(\phi_h,\varphi_h)=a_h(\varphi_h,\phi_h)=(\na A_h \varphi_h,\na\phi_h)_\Om=(\na\phi_h,\na A_h \varphi_h)_\Om.
$$
We can establish the following stability result
\begin{lemma}\label{lemma: elliptic na Ah}
Let $g\in H^1_0(\Om)$ and $w_h\in V_h$ be the finite element solution
$$
a_h(w_h,\varphi_h)=(\na g,\na\varphi_h)_\Om\quad \textforall\varphi_h\in V_h.
$$  
Then,
$$
\|\na A_h w_h\|_{L^2(\Om)}\le \|\na g\|_{L^2(\Om)}.
$$
\end{lemma}
\begin{proof}
By the definition of $A_h$ and the equation for $w_h$,
$$
\|\na A_h w_h\|_{\Om}^2=(\na A_h w_h, \na A_h w_h)_\Omega=a_h( w_h, A_h w_h)=(\na g,\na A_h w_h)_\Om\le \|\na g\|_{\Om}\|\na A_h w_h\|_{\Om}.
$$
Canceling, we obtain the result.
\end{proof}

We wish our analysis to be applicable to a broad range of discretizations, for instance to the 
  $C^0$-interior penalty method (cf. \cite{BrennerS_SungLY_2005}), in which case $a_h(\psi,\cdot)$ is not well
  defined for general $\psi \in H^2_0(\Om)$. 
  Vice versa, not for all data $f \in H^{-2}(\Om)$ the right-hand side can be tested with $\phi_h \in V_h$.
  Hence, throughout the paper, we work under the following assumptions, the first ensuring well-posedness 
  of the discrete problems, and the second guaranteeing an elliptic Galerkin orthogonality.
    \begin{assumption}\label{ass:Vh_contained}
        There exists $\theta \in (0,1]$, such that
            $V_h \subset H^{1+\theta}_0(\Om)$. 
    \end{assumption}
    \begin{assumption}\label{ass:discrete_bilin_form}
      Let $\theta$ and $V_h$ be as in \Cref{ass:Vh_contained}.
      The bilinear form $a_h(\cdot,\cdot)$ can be extended to 
      $$ a_h(\cdot,\cdot):(X+V_h) \times (X+V_h) \to \mathbb{R}, $$
      where $X:= \mathcal A^{-1}(H^{-1-\theta}(\Om))$. Furthermore, for all $\psi \in X, \varphi_h \in V_h$,
      there holds
      \begin{equation*}
          \Opair{\mathcal A \psi,\varphi_h} = a_h(\psi,\varphi_h),
      \end{equation*}
      where $\mathcal A$ is the biharmonic operator introduced 
      in \eqref{eq:def_biharmonic_operator}.
    \end{assumption}
    \begin{remark}\label{remark: assumptions}
        Although technical in nature, these assumptions are not overly restrictive, mainly posing a regularity
        assumption on the solution $\psi$ of \eqref{eq:Fourth}.
        In \cite[Eq. 1.4]{BrennerS_SungLY_2005} the authors show that on any polygonal domain, 
        for the choice $\theta = \frac12 - \varepsilon$, where $\varepsilon > 0$ can be taken 
        arbitrarily small, there holds $X = H^2_0(\Om) \cap H^{\frac{5}{2} + \varepsilon}(\Om)$.
        For this choice of $\theta$, taking $V_h$ as a standard $C^0$-finite element space trivially 
        satisfies \Cref{ass:Vh_contained}. In \cite[Lemma 5]{BrennerS_SungLY_2005} it is then shown, that 
        if $a_h(\cdot,\cdot)$ is chosen as the interior penalty bilinear form then \Cref{ass:discrete_bilin_form}
        holds.
        On the other hand, for $H^2_0(\Om) \supset V_h$ the $C^1$-finite element space, together with
        $a_h(\cdot,\cdot) = \Oprod{\Delta \cdot,\Delta \cdot}$, 
        Assumptions \ref{ass:Vh_contained} and \ref{ass:discrete_bilin_form} 
        are satisfied for $\theta = 1$, $X=H^2_0(\Om)$.
    \end{remark}
 Under \Cref{ass:discrete_bilin_form}, we have the following form of a Galerkin orthogonality:
     Let $\psi \in H^2_0(\Om)$ be the solution of \eqref{eq:Fourth} for given data $f \in H^{-1-\theta}(\Om)$.
     Then $\psi \in X$, and it holds
    \begin{equation}\label{eq:elliptic_galerkin_orth}
         a_h(\psi-\psi_h,\varphi_h) = 0 \qquad \textforall \varphi_h \in V_h,
     \end{equation}
     where $\psi_h \in V_h$ is the solution to the discrete weak formulation of \eqref{eq:Fourth} satisfying
     \begin{equation*}
         a_h(\psi_h,\varphi_h) = \Opair{f,\varphi_h} \qquad \textforall \varphi_h \in V_h.
     \end{equation*}
 Given such a bilinear form, we define the (biharmonic) Ritz projection
 $R_h\colon{X}\to V_h, w \mapsto R_h w$, with $X$ from Assumption \ref{ass:discrete_bilin_form}, by
 \begin{equation}\label{eq:def_ritz_projection}
    a_h(w-R_hw,\chi)=0\quad \textforall\chi\in V_h.
\end{equation}

% \subsection{Temporal discretization: the discontinuous Galerkin method}
% \label{sec: time discretization}
  \subsection{Full space-time discretization}
\label{sec: full discretization}  
To discretize in time, 
we use the discontinuous Galerkin method.
%In this section we introduce  the discontinuous Galerkin method for the time discretization.
For that, we partition $I = (0,T]$ into subintervals $I_m = (t_{m-1},t_m]$ of length $k_m = t_m - t_{m-1}$, where $0= t_0<t_1<\dots <t_{M-1}<t_M = T$. The maximal and minimal time steps are denoted by $k=\max_m k_m$ and $k_{\min}=\min_mk_m$, respectively.  
We use the following standard notation for a function $u_k \in X_{k}^r(V)$, for a Banach space $V$, where
$$
X_{k}^r(V):=\set{u_k\in L^2(I;V)\ |\ u_k\mid_{I_m}\in \mathbb{P}_r(I_m;V),\ m=1,2,\dots,M },
$$
\begin{equation}
u_{k,m}^+ = \lim_{\varepsilon \rightarrow 0^+}u_k(t_m+\varepsilon), \quad u_{k,m}^- = \lim_{\varepsilon \rightarrow 0^+}u_k(t_m-\varepsilon),\quad [u_k]_m = u_{k,m}^+-u_{k,m}^-.
\end{equation}
In several places of the argument we will use the following identity for any $w(t)$ 
with well defined 
one-sided limits at $t=t_{m-1}$, there holds
\begin{equation}\label{eq:jump_identity}
([w]_{m-1}, w_{m-1}^+)_{\Omega}= \frac{1}{2}\|w_{m-1}^+\|^2_{\Om}+\frac{1}{2}\| [w]_{m-1}\|^2_{\Om}-\frac{1}{2}\|w_{m-1}^-\|^2_{\Om}.
\end{equation}
We define a time projection $\pi_k\colon C(\bar{I})\to X_k^r(\R)$ which on each time interval $I_m$ is defined by
\begin{equation}\label{eq: projection pik}
\begin{aligned}
(\pi_k v-v,\phi)_{I_m}&=0\quad \textforall \phi\in \mathbb{P}_{r-1}(I_m),\quad r>0,\\
\pi_k v(t_m^-)&=v(t_m^-).
\end{aligned}
\end{equation}
In the case of $dG(0)$ method, the projection is defined by the second condition only. 
With these considerations, we can now define the fully discrete approximation to the problem
  \eqref{eq:transient_biharmonic_weak}, by introducing the fully discrete bilinear form
  $B_h\colon X_{k}^r(V_h)\times X_{k}^r(V_h)\to \mathbb{R}$
\begin{equation}\label{eq: fully discrete bilinear form}
  \begin{aligned}
    {B}_h(u_{kh},v_{kh}) = & \sum_{m=1}^M (\partial_t \nabla u_{kh}, \nabla v_{kh} )_{I_m\times \Omega} 
   + \int_0^Ta_h( u_{kh}, v_{kh}) dt\\
   & + \sum_{m=2}^M ([\nabla u_{kh}]_{m-1},\na v_{kh,m-1}^+)_{\Omega} + (\na u_{kh,0}^+,\na v_{kh,0}^+)_{\Omega}.
 \end{aligned}
\end{equation}
Then the fully discrete problem reads: Find  $\psi_{kh}\in X_{k}^r(V_h)$, 
such that for given $f\in L^2(I;H^{-1-\theta}(\Om))$, $\psi_0 \in H^{2-\theta}_0(\Om)$, it holds
\begin{equation}\label{eq:fully_discrete_instationary}
    {B}_h(\psi_{kh},\varphi_{kh})=\IOpair{f, \varphi_{kh}}+\Oprod{\na \psi_0,\na \varphi_{kh,0}^+},
    \quad \text{for all}\ \varphi_{kh}\in X_{k}^r(V_h),
\end{equation}
where $\theta \in (0,1]$ satisfies \Cref{ass:Vh_contained}.
    Note that due to \Cref{ass:Vh_contained},
    the right-hand side $\IOpair{f,\varphi_{kh}}$ is well defined. Moreover, we have the following parabolic Galerkin orthogonality
between the continuous and fully discrete solution. 
\begin{lemma}
  Let $\theta$ satisfy the Assumptions \ref{ass:Vh_contained} and \ref{ass:discrete_bilin_form}
  and let $f \in L^2(I;H^{-1-\theta}(\Om))$,
    $\psi_0 \in H^{2-\theta}_0(\Om)$.
    Further let $\psi$, $\psi_{kh}$ be the continuous and fully discrete solutions to
    \eqref{eq:transient_biharmonic_weak} and \eqref{eq:fully_discrete_instationary} respectively.
    Then for any $\varphi_{kh} \in X^r_k(V_h)$ there holds 
    \begin{equation}\label{eq: Galerkin orthogonality fully discrete}
        B_h(\psi-\psi_{kh},\varphi_{kh}) = 0.
    \end{equation}
\end{lemma}
\begin{proof}
    Let $\varphi_{kh} \in X_k^r(V_h)$ be arbitrary.
    Due to the regularity of $f$ and $\psi_0$, as a consequence of the interpolated regularity result in
    \Cref{corr:parabolic_regularity_interpolated}, we have 
    \begin{equation}\label{eq: regularity psi}
        \partial_t \Delta \psi, \mathcal A \psi \in L^2(I;H^{-1-\theta}(\Om)).
    \end{equation}
    By \Cref{ass:discrete_bilin_form},  the regularity of $\mathcal A \psi$
    implies $\psi \in L^2(I;X)$, with $X$ from \Cref{ass:discrete_bilin_form}, and thus 
    in the definition of $B_h(\cdot,\cdot)$, the term $\int_I a_h(\psi,\varphi_{kh}) \, d t$ is well defined.
    Further, as due to \Cref{thm:existence_uniqueness_standard_regularity} it holds 
    $\nabla \psi \in C(\bar I;L^2(\Om)^2)$, all jump terms vanish and 
    $\psi_0^+ = \lim_{t\to 0 +} \psi(t) = \psi(0)$.
    We thus obtain
    \begin{equation*}
        {B}_h(\psi,\varphi_{kh}) = 
        \int_0^T \Oprod{\partial_t \nabla \psi, \nabla \varphi_{kh} }
        + a_h( \psi, \varphi_{kh}) dt
        + \Oprod{\na \psi(0),\na \varphi_{kh,0}^+}.
    \end{equation*}
    By \Cref{ass:discrete_bilin_form},
    due to $\psi \in L^2(I;X)$ we can replace the discrete elliptic bilinear form 
    $a_h(\psi,\varphi_{kh})$ by the continuous elliptic operator $\Opair{\mathcal A \psi,\varphi_{kh}}$.
    Together with the definition of the Laplace operator, this yields
    \begin{equation*}
        {B}_h(\psi,\varphi_{kh}) = 
        -\IOpair{\partial_t \Delta \psi, \varphi_{kh} }
        + \IOpair{\mathcal A\psi, \varphi_{kh}}
        + \Oprod{\na \psi(0),\na \varphi_{kh,0}^+}.
    \end{equation*}
    As by \Cref{ass:Vh_contained},
    $V_h \subset H^{1+\theta}_0(\Om)$, and $H^2_0(\Om)$ is dense in $H^{1+\theta}_0(\Om)$,
    we can approximate $\varphi_{kh}$ by a sequence $\{\varphi^{kh}_n\} \subset L^2(I;H^2_0(\Om))$ with
    $\varphi^{kh}_n \to \varphi_{kh}$ in $L^2(I;H^{1+\theta}_0(\Om))$.
    As due to \eqref{eq:transient_biharmonic_weak}, it holds 
    \begin{equation*}
        -\IOpair{\partial_t \Delta \psi, \varphi^{kh}_n} + \IOpair{\mathcal A\psi, \varphi^{kh}_n}
        = \IOpair{f,\varphi^{kh}_n},
    \end{equation*}
    letting $n \to \infty$, the regularities of $f$ given by assumption, as well as 
    $\partial_t \Delta \psi, \mathcal A\psi$ from \eqref{eq: regularity psi} yield
    \begin{equation*}
        -\IOpair{\partial_t \Delta \psi, \varphi_{kh}} + \IOpair{\mathcal A\psi, \varphi_{kh}}
        = \IOpair{f,\varphi_{kh}}.
    \end{equation*}
    Hence together with $\psi(0) = \psi_0$
    \begin{equation*}
        {B}_h(\psi,\varphi_{kh}) = 
        \IOpair{f, \varphi_{kh} } + \Oprod{\na \psi_0,\na \varphi_{kh,0}^+}
        = {B}_h(\psi_{kh},\varphi_{kh}),
    \end{equation*}
    where the last step is precisely the definition of $\psi_{kh}$. 
    As $\varphi_{kh}$ was arbitrary, this concludes the proof.
\end{proof}
%\begin{remark}
%    Note, that the parabolic Galerkin %orthogonality in the above result is a direct %consequence of the
%    elliptic Galerkin orthogonality %\eqref{eq:elliptic_galerkin_orth}.
%    For the time discretization, due to the property $\nabla \psi \in C(I;L^2(\Om)^2)$, shown in 
%    \Cref{thm:existence_uniqueness_standard_regularity}, the jump terms in time vanish and the intergral 
%    over the time derivative coincides with the term of the continuous formulation 
%    \eqref{eq:transient_biharmonic_weak}.
%\end{remark}

For the fully discrete bilinear form, the following dual expression can be shown via integration by parts on each
$I_m$ and rearranging of terms:
\begin{equation}\label{eq:fully_discrete_bilinear_form_dual}
  \begin{aligned}
    {B}_h(u_{kh},v_{kh}) = & -\sum_{m=1}^M ( \nabla u_{kh}, \partial_t\nabla v_{kh} )_{I_m\times \Omega} 
        + \int_0^T a_h(u_{kh},v_{kh}) \ dt\\
    & - \sum_{m=1}^{M-1} \Oprod{\nabla u^-_{{kh},m},[\na v_{kh}]_{m}} 
        + (\na u_{{kh},M}^-,\na v_{{kh},M}^-)_{\Omega}.
  \end{aligned}
\end{equation}

\section{Fully discrete error analysis}\label{sec:error_analysis}
First, we provide  the identity analogous to \eqref{eq:jump_identity} for the symmetric form
$a_h(\cdot, \cdot)$, namely
\begin{equation}\label{eq:jump_identity ah}
a_h([w]_{m-1}, w_{m-1}^+)=\frac{1}{2}a_h(w^+_{m-1}, w_{m-1}^+)+\frac{1}{2}a_h([w]_{m-1}, [w]_{m-1})-\frac{1}{2}a_h(w_{m-1}^-, w_{m-1}^-),
\end{equation}
for any $w(t)\in V_h$. The identity follows by adding and subtracting $\frac{1}{2}a_h([w]_{m-1}, [w]_{m-1})$
followed by direct calculations.
We moreover introduce the notation
\begin{equation}\label{eq:jump_at_0_convention}
    \psi_{kh,0}^- := \Pi_h \psi_0, \text{ and correspondingly } [\psi_{kh}]_0 = \psi_{kh,0}^+ - \Pi_h \psi_0,
\end{equation}
where $\Pi_h$ is the $H^1_0$-projection from \eqref{eq:definition_Pih}.
 
Before we show error estimates for the discrete problem, we shall first discuss stability estimates.
  Analogously to the continuous results \Cref{thm:existence_uniqueness_standard_regularity,prop:weak_stability}
  we are interested in two different stability estimates.
  One for low regularity of data (as low as possible such that the discrete problem still is well defined),
  and one for higher regularity. In order to derive error estimates by a duality approach, we only require 
  a stability estimate for higher regularity of data, which we will apply for an appropriate dual problem.
  This estimate can be shown under the assumptions made so far, and thus we present it first in
  \Cref{lem: stability fully}, which is the discrete analog of \Cref{prop:weak_stability}.
  Afterwards, in \Cref{thm:stability_low_regularity} we shall discuss how 
  \Cref{thm:existence_uniqueness_standard_regularity} has to be generalized to the discrete case.
The following proof follows the lines of \cite{MeidnerVexler:2008I} for second order parabolic problems. 
\begin{theorem}\label{lem: stability fully}
  Let $g \in L^2(I;H^1(\Om)) ,\psi_0 \in H^1_0(\Om)$, and let $\psi_{kh} \in X^r_k(V_h)$ be the 
    fully discrete solution to
  \begin{equation}
      {B}_h(\psi_{kh},\varphi_{kh})=\IOprod{\nabla g, \nabla \varphi_{kh}}
      +\Oprod{\na \psi_0,\na \varphi_{kh,0}^+}
      \quad \text{for all $\varphi_{kh}\in X_{k}^r(V_h)$}.
  \end{equation}
  Then there exists a constant $C$ independent of $k$ and $h$ such that 
$$
\sum_{m=1}^M\|\na\partial_t \psi_{kh}\|^2_{I_m\times \Om}+\|\na A_h \psi_{kh}\|^2_{I\times \Om}+\sum_{m=1}^M k^{-1}_m\|[\na \psi_{kh}]_{m-1}\|^2_\Omega + \revB{\vertiii{\psi_{kh,M}^-}^2_h}\le C (\vertiii{\Pi_h \psi_0}^2_h + \|\na g\|^2_{I\times \Om}).
$$  
\end{theorem}
\begin{proof}
  First we test the equation for each $m=1,2,...,M$ with $\varphi_{kh}|_{I_m} =A_h\psi_{kh}|_{I_m}$ and set 
  $\varphi_{kh}$ to zero on all other subintervals. Then on each time subinterval $I_m$, we obtain
$$
\int_{I_m}(\na \partial_t \psi_{kh},\na A_h\psi_{kh} )_\Omega dt+\int_{I_m}\|\na A_h \psi_{kh}\|^2dt +(\na [\psi_{kh}]_{m-1}, \na A_h \psi_{kh,m-1}^+)_\Om=\int_{I_m}(\na g,\na A_h \psi_{kh})_\Omega dt,
$$
where for $m=1$, we have made use of our convention \eqref{eq:jump_at_0_convention}.
% Using that
We may rewrite this, using
$$
(\partial_t \na \psi_{kh},\na A_h \psi_{kh})_{\ImOm} =\int_{I_m}\frac{1}{2}\frac{d}{dt}a_h(\psi_{kh},\psi_{kh})\ dt = \frac{1}{2}a_h(\psi_{kh,m}^-,\psi_{kh,m}^-)-\frac{1}{2}a_h(\psi_{kh,m-1}^+,\psi_{kh,m-1}^+)
$$
and
%since 
$$
(\na [\psi_{kh}]_{m-1}, \na A_h \psi_{kh,m-1}^+)_\Om=a_h([\psi_{kh}]_{m-1},\psi^+_{kh,m-1}),
$$
together with the identity \eqref{eq:jump_identity ah}.
%applying the identity \eqref{eq:jump_identity ah},
By noticing that the term with $a_h(\psi_{kh,m-1}^+,\psi_{kh,m-1}^+)$ cancels out, as a result
we obtain
$$
\begin{aligned}
\frac{1}{2}a_h(\psi_{kh,m}^-,\psi_{kh,m}^-)-\frac{1}{2}a_h(\psi_{kh,m-1}^-,\psi_{kh,m-1}^-)+\frac{1}{2}a_h([\psi_{kh}]_{m-1},[\psi_{kh}]_{m-1}) &+\int_{I_m}\|\na A_h \psi_{kh}\|_\Om^2dt\\&=(\na g, \na A_h \psi_{kh})_{\ImOm}.
\end{aligned}
$$
Summing over $m$, and using the convention $\psi_{kh,0}^- = \Pi_h \psi_0$ yields
$$
\begin{aligned}
\frac{1}{2}a_h(\psi_{kh,M}^-,\psi_{kh,M}^-)+\frac{1}{2}\sum_{m=1}^Ma_h([\psi_{kh}]_{m-1},[\psi_{kh}]_{m-1})&+\int_I \|\na A_h \psi_{kh}\|_\Om^2dt\\
&=(\na g,\na A_h\psi_{kh})_{\IOm}+\frac{1}{2}a_h(\Pi_h\psi_0,\Pi_h\psi_0).
\end{aligned}
$$
Using the Cauchy-Schwarz and Young's inequalities, we have
\revB{
$$
\begin{aligned}
\frac{1}{2} \vertiii{\psi_{kh,M}^-}^2_h + \|\na A_h \psi_{kh}\|^2_{\IOm} &\le \|\na g\|_{I\times \Om}\|\na A_h\psi_{kh}\|_{\IOm}+\frac{1}{2}\vertiii{\Pi_h\psi_0}_h^2\\
&\le \frac{1}{2}\|\na g\|_{I\times \Om}^2+\frac{1}{2}\|\na A_h\psi_{kh}\|^2_{\IOm}+\frac{1}{2}\vertiii{\Pi_h\psi_0}_h^2.
\end{aligned}
 $$
}
Canceling, we obtain
\begin{equation}\label{eq: estimates for integral of ah}
\revB{\vertiii{\psi_{kh,M}^-}^2_h}+ \|\na A_h \psi_{kh}\|^2_{\IOm}\le \|\na g\|^2_{I\times \Om}+\vertiii{\Pi_h\psi_0}_h^2.
\end{equation}

{\it Step 2. Time derivative terms.}

Note that for the case $r=0$ there is nothing to prove here, hence we assume in this step $r\ge 1$.
Noticing that $(t-t_{m-1})\partial_t \psi_{kh}\in \mathbb{P}_r(I_m;V_h)$, we  choose $\varphi_{kh}=(t-t_{m-1})\partial_t \psi_{kh}$ on the time interval $I_m$ and zero on others. Noticing that the term with jumps vanishes, we have
$$
\int_{I_m}(t-t_{m-1})\|\na \partial_t\psi_{kh}\|^2_\Omega dt+\int_{I_m}(t-t_{m-1})a_h(\psi_{kh},\partial_t \psi_{kh})dt =\int_{I_m}(t-t_{m-1})(\na g,\na \partial_t\psi_{kh})_\Omega dt.
$$
For the second term, using the Cauchy-Schwarz and Young's inequalities,
$$
\begin{aligned}
\int_{I_m}(t-t_{m-1})a_h(\psi_{kh},\partial_t \psi_{kh})dt&=\int_{I_m}(t-t_{m-1})( \na \partial_t\psi_{kh},\na A_h \psi_{kh})_\Omega dt\\
&\le \left(\int_{I_m}(t-t_{m-1})\| \na\partial_t \psi_{kh}\|_\Om^2dt\right)^{1/2}\left(\int_{I_m}(t-t_{m-1})\| \na A_h\psi_{kh}\|_\Om^2dt\right)^{1/2}\\
&\le \frac{1}{4}\int_{I_m}(t-t_{m-1})\| \na \partial_t \psi_{kh}\|_\Om^2dt+\int_{I_m}(t-t_{m-1})\| \na A_h\psi_{kh}\|_\Om^2dt.
\end{aligned}
$$
Similarly, 
using the Cauchy-Schwarz and Young's inequalities,
$$
\begin{aligned}
\int_{I_m}(t-t_{m-1})(\na g,\na\partial_t \psi_{kh})_\Omega dt&\le \left(\int_{I_m}(t-t_{m-1})\|\na g\|^2_\Omega dt\right)^{1/2}\left(\int_{I_m}(t-t_{m-1})\|\na \partial_t\psi_{kh}\|^2_\Omega dt\right)^{1/2}\\
&\le \frac{1}{4}\int_{I_m}(t-t_{m-1})\| \na\partial_t \psi_{kh}\|_\Om^2dt+\int_{I_m}(t-t_{m-1})\| \na g\|^2_\Om dt.
\end{aligned}
$$
Combining estimates, and kicking back, we obtain
$$
\int_{I_m}(t-t_{m-1})\|\na \partial_t\psi_{kh}\|^2_\Omega dt\le C\int_{I_m}(t-t_{m-1})\|\na g\|^2_\Omega dt+C\int_{I_m}(t-t_{m-1})\| \na A_h\psi_{kh}\|^2_\Om dt.
$$
Using the inverse inequality
$$
 \int_{I_m}\|\chi_{k}\|^2_\Omega dt\le Ck_{m}^{-1}\int_{I_m}(t-t_{m-1})\|\chi_{k}\|^2_\Omega dt
$$
which holds for any piecewise polynomial in time with fixed maximal degree,
for example for $\chi_k=\na\partial_t \psi_{kh}$, we have
$$
\begin{aligned}
\int_{I_m}\|\na \partial_t\psi_{kh}\|^2_\Omega dt&\le  Ck_{m}^{-1}\int_{I_m}(t-t_{m-1})\|\na g\|^2_\Omega dt+Ck_{m}^{-1}\int_{I_m}(t-t_{m-1})\| \na A_h\psi_{kh}\|_\Om^2dt\\
&\le  C\int_{I_m}\|\na g\|^2_\Omega dt+C\int_{I_m}\| \na A_h\psi_{kh}\|_\Om^2dt.
\end{aligned}
$$
Summing over $m$ and using \eqref{eq: estimates for integral of ah} yields
\begin{equation}\label{eq: estimates for integral of pa_tvkh}
  \int_{I}\|\na \partial_t\psi_{kh}\|^2_\Omega dt\le C(\vertiii{\Pi_h \psi_0}^2_h + \|\na g\|^2_{I\times \Om}).
\end{equation}

{\it Step 3. Jump terms.}

For every $m=1,...,M$, we choose $\varphi_{kh}|_{I_m} =[\psi_{kh}]_{m-1}$ and set $\varphi_{kh}$
to zero on all other subintervals. Inserting this choice yields
$$
\int_{I_m}(\na\partial_t\psi_{kh},[\na \psi_{kh}]_{m-1})_\Om dt+\int_{I_m}(\na A_h\psi_{kh},[\na \psi_{kh}]_{m-1})_\Om dt +\|[\na \psi_{kh}]_{m-1}\|^2_\Omega =\int_{I_m} (\na g, [\na \psi_{kh}]_{m-1})_\Om dt.
$$
Using the Cauchy-Schwarz and Young's inequalities,
$$
\begin{aligned}
\int_{I_m}(\na\partial_t\psi_{kh},[\na \psi_{kh}]_{m-1})_\Om dt&\le 
\left(k_m\int_{I_m}\|\na\partial_t\psi_{kh}\|^2_\Omega dt\right)^{1/2}\left(k^{-1}_m\int_{I_m}\|[\na \psi_{kh}]_{m-1} \|^2_\Omega dt\right)^{1/2}\\
&\le 
k_m\int_{I_m}\|\na\partial_t\psi_{kh}\|^2_\Omega dt+\frac{1}{4}\|[\na \psi_{kh}]_{m-1} \|^2_\Omega,
\end{aligned}
$$
where we used that $\varphi_{kh}|_{I_m} = [\psi_{kh}]_{m-1}$ is constant on $I_m$. Similarly,
$$
\begin{aligned}
\int_{I_m}(\na A_h\psi_{kh},[\na \psi_{kh}]_{m-1})_\Om dt&\le 
\left(k_m\int_{I_m}\|\na A_h\psi_{kh}\|^2_\Omega dt\right)^{1/2}\left(k^{-1}_m\int_{I_m}\|[\na \psi_{kh}]_{m-1} \|^2_\Omega dt\right)^{1/2}\\
&\le 
k_m\int_{I_m}\|\na A_h\psi_{kh}\|^2_\Omega dt+\frac{1}{4}\|[\na \psi_{kh}]_{m-1} \|^2_\Omega,
\end{aligned}
$$
and
$$
\begin{aligned}
  \int_{I_m} (\na g, [\na \psi_{kh}]_{m-1})_\Om dt&\le \left(k_m\int_{I_m}\|\na g\|^2_\Omega dt\right)^{1/2}\left(k^{-1}_m\int_{I_m}\|[\na \psi_{kh}]_{m-1} \|^2_\Omega dt\right)^{1/2}\\
                                             &\le k_m\int_{I_m}\|\na g\|^2_\Omega dt+\frac{1}{4}\|[\na \psi_{kh}]_{m-1} \|^2_\Omega.
\end{aligned}
$$
Combining, we obtain
$$
k^{-1}_m\|[\na \psi_{kh}]_{m-1}\|^2_\Omega \le C\left(\int_{I_m}\|\na g\|^2_\Omega dt+\int_{I_m}\|\na\partial_t\psi_{kh}\|^2_\Omega dt+\int_{I_m}\|\na A_h\psi_{kh}\|^2_\Omega dt\right).
$$ 
Summing over $m$ and using \eqref{eq: estimates for integral of ah} and \eqref{eq: estimates for integral of pa_tvkh}, we obtain
$$
\sum_{m=1}^Mk^{-1}_m\|[\na \psi_{kh}]_{m-1}\|^2_\Omega 
\le C(\vertiii{\Pi_h \psi_0}^2_h + \|\na g\|^2_{I\times \Om}),
$$
which completes the proof of the lemma. 
\end{proof}

After having shown the discrete stability result for higher regularity data, let us discuss 
  the lower regularity setting. Note that we only include this result for completeness, and do not 
  require it for the error estimates.
  \begin{theorem}\label{thm:stability_low_regularity}
    Let \Cref{ass:Vh_contained} be satisfied, and let
    $f\in L^2(I;H^{-1-\theta}(\Om))$, $\psi_0 \in H^1_0(\Om)$.
  Let $\psi_{kh} \in X^r_k(V_h)$ solve \eqref{eq:fully_discrete_instationary}. Then we have the following.
  If there exists a constant $\widetilde C>0$ independent of $h$, such that 
  \begin{equation}\label{eq:ass_stat_stab}
    \|\varphi_h\|_{H^{1+\theta}_0(\Om)} \le \widetilde C \vertiii{\varphi_h}_h \qquad \text{for all } \varphi_h \in V_h,
  \end{equation}
  then there exists a constant $C>0$ independent of $k,h$ such that:
  \begin{equation*}
    \int_I \vertiii{\psikh}^2_h \ dt + \sum_{m=2}^M \|\na [\psi_{kh}]_{m-1}\|_\Om^2 
    +  \max_{1\le m \le M} \|\na \psi_{kh,m}^-\|_\Om^2
    \le C \|f\|_{L^2(I;H^{-1-\theta}(\Om))}^2 + \|\na \psi_0\|_\Om^2 .
  \end{equation*}
\end{theorem}
\begin{proof}
  We begin by testing \eqref{eq:fully_discrete_instationary} with $\psi_{kh}$. For the left hand side, we obtain
  with the fundamental theorem of calculus, and the identity \eqref{eq:jump_identity} applied 
  to the jump term after canceling:
  \begin{equation}\label{eq:discrete_stability_tested}
    \begin{aligned}
    B_h(\psikh,\psikh)  &= \sum_{m=1}^M \ImOprod{\nabla \pa_t \psikh,\na\psikh} + \int_I a_h(\psikh,\psikh) \ dt\\
                        & \quad + \sum_{m=2}^M \Oprod{\na [\psikh]_{m-1},\na \psi_{kh,m-1}^+} 
                            + \Oprod{\na \psi_{kh,0}^+,\na \psi_{kh,0}^+}\\
                        &= \sum_{m=1}^M \left[\frac12 \|\na \psi_{kh,m}^-\|_\Om^2 - \frac12 \|\na \psi_{kh,m-1}^+\|_\Om^2\right]
                             + \int_I \vertiii{\psikh}^2_h \ dt\\
                        & \quad + \frac12 \sum_{m=2}^M \left[\|\na \psi_{kh,m-1}^+\|_\Om^2 + \|\na [\psi_{kh}]_{m-1}\|_\Om^2 - \|\na \psi_{kh,m-1}^-\|_\Om^2 \right]
                            + \|\na \psi_{kh,0}^+\|_\Om^2\\
                        &=\int_I \vertiii{\psikh}^2_h \ dt + \frac12 \sum_{m=2}^M \|\na [\psi_{kh}]_{m-1}\|_\Om^2 
                          + \frac12 \|\na \psi_{kh,0}^+\|_\Om^2+ \frac12 \|\na \psi_{kh,M}^-\|_\Om^2.
  \end{aligned}
  \end{equation}
  For the right-hand side, using \eqref{eq:ass_stat_stab}, we obtain together with Young's inequality
  \begin{align*}
    \IOpair{f,\psikh} + \Oprod{\na \psi_0,\na \psi_{kh,0}^+}
    &\le \|f\|_{L^2(I;H^{-1-\theta}(\Om))}\|\psikh\|_{L^2(I;H^{1+\theta}(\Om))} + \|\na \psi_0\|_\Om \|\na \psi_{kh,0}^+\|_\Om\\
    &\le C \|f\|_{L^2(I;H^{-1-\theta}(\Om))}^2 + \frac12 \int_I \vertiii{\psikh}_h^2 dt + \frac12 \|\na \psi_0\|_\Om^2 
      + \frac12 \|\na \psi_{kh,0}^+\|_\Om^2.
  \end{align*}
  We can absorb the second and fourth term into the last line of \eqref{eq:discrete_stability_tested}
  and multiply by 2, to obtain
  \begin{equation*}
    \int_I \vertiii{\psikh}^2_h \ dt + \sum_{m=2}^M \|\na [\psi_{kh}]_{m-1}\|_\Om^2 
    + \|\na \psi_{kh,M}^-\|_\Om^2
    \le C \|f\|_{L^2(I;H^{-1-\theta}(\Om))}^2 +  \|\na \psi_0\|_\Om^2 .
  \end{equation*}
  This is almost the claimed result. If we test \eqref{eq:fully_discrete_instationary}
  not with $\psikh$, but with a function which coincides with $\psikh$ on $I_1 \cup ... \cup I_m$
  and is zero on the remaining intervals, we get the final result by exactly the same argument.
\end{proof}
\begin{remark}
  It is straightforward to see, that the additional condition \eqref{eq:ass_stat_stab} is equivalent to 
  a stability estimate 
  $
    \vertiii{\varphi_h}_h \le C \|f\|_{H^{-1-\theta}(\Om)},
    $
  for all $f \in H^{-1-\theta}(\Om)$, where $\varphi_h$ solves the stationary problem
  \begin{equation}
    a_h(\varphi_h,\phi_h) = \Opair{f,\phi_h} \qquad \text{for all } \phi_h \in V_h.
  \end{equation}
\end{remark}
\begin{remark}
  For $C^1$ finite elements satisfying $V_h \subset H^2_0(\Om)$ the condition \eqref{eq:ass_stat_stab} 
  is trivially satisfied.
  For the standard $C^0$ finite elements used in the interior penalty method of 
  \cite{BrennerS_SungLY_2005} it necessarily has to hold $\theta < \frac12$.
  One may then employ the arguments of \cite{belgacem_nonstandard_2001}, to obtain a localized expression
  of the $H^{1+\theta}_0(\Om)$ norm:
  \begin{equation*}
    \|\varphi_h\|_{H^{1+\theta}(\Om)}^2 \le C_\theta \sum_{T \in \mathcal T_h} \|\varphi_h\|_{H^{1+\theta}(T)}^2.
  \end{equation*}
  Each of the local terms $\|\varphi_h\|_{H^{1+\theta}(T)}^2$ can be bounded by the corresponding
  term $\|\varphi_h\|_{H^{2}(T)}^2$ contained in $\vertiii{\varphi_h}_h^2$.
\end{remark}
We are now able to formulate our main result of the paper.

\begin{theorem}[Main result]\label{thm:error fully discrete}
  Let the Assumptions \ref{ass:Vh_contained} and \ref{ass:discrete_bilin_form} be fulfilled.
  Let $\psi \in \XX$ and $\psi_{kh} \in X_k^r(V_h)$ be the exact and the fully discrete Galerkin solutions satisfying \eqref{eq:transient_biharmonic_weak} and \eqref{eq:fully_discrete_instationary} respectively 
  for given data $f \in L^2(I;H^{-1-\theta}(\Om))$, $\psi_0 \in H^{2-\theta}_0(\Om)$.
  Then, there exists a constant $C$ independent of $k$ and $h$ such that  for any $\chi_{kh}\in X^r_k(V_h)$
\begin{equation}\label{eq:error fully discrete}
    \norm{\na(\psi-\psi_{kh})}_{I\times\Omega} 
    \leq C  \Big(\norm{\na(\psi-\chi_{kh})}_{I\times\Omega}
    +\norm{\na (R_h \psi - \psi)}_{I\times\Omega} 
    +\norm{\na (\pi_k\psi - \psi)}_{I\times\Omega} \Big).
\end{equation}
\end{theorem}
\begin{proof} 
  Let $\chi_{kh} \in X_k^r(V_h)$ be arbitrary but fixed. 
    We begin by showing the estimate
\begin{equation}\label{eq:error_fully_discrete_chikh}
    \norm{\na(\psi_{kh} - \chi_{kh})}_{I\times\Omega} 
    \leq C  \Big(\norm{\na(\psi-\chi_{kh})}_{I\times\Omega}
    +\norm{\na (R_h \psi - \chi_{kh})}_{I\times\Omega} 
    +\norm{\na (\pi_k\psi - \chi_{kh})}_{I\times\Omega} \Big).
\end{equation}
  To this end we consider the dual problem
\begin{equation}\label{eq: dual problem}
  B_h(\varphi_{kh},z_{kh})=(\na (\psi_{kh} - \chi_{kh}),\na \varphi_{kh})_{I\times\Omega}\quad \textforall\varphi_{kh}\in X_{k}^r(V_h).
\end{equation}
Note that due to this definition, $z_{kh}$ has a zero terminal condition at $t=T$. Hence,
analogously to \eqref{eq:jump_at_0_convention}, we use the abbreviating notation $[z_{kh}]_M = 0 - z_{kh,M}^-$.
Then, due to the assumption on the data, by the Galerkin orthogonality 
\eqref{eq: Galerkin orthogonality fully discrete} and using the dual expression of the bilinear form
\eqref{eq:fully_discrete_bilinear_form_dual},
$$
\begin{aligned}
  \|\na (\psi_{kh}-\chi_{kh})\|^2_{I\times\Omega}
  &=B_h(\psi_{kh} - \chi_{kh},z_{kh})=B_h(\psi-\chi_{kh},z_{kh})\\
  &=-\sum_{m=1}^M \int_{I_m}(\na (\psi - \chi_{kh}),\na \partial_t z_{kh})_\Omega \ dt
  +\int_I a_h(\psi - \chi_{kh},z_{kh})dt \\
  & \qquad -\sum_{m=1}^{M}(\na (\psi(t_m)-\chi_{kh}(t_m)),[\na z_{kh}]_m)_\Omega\\
&:=J_1+J_2+J_3.
\end{aligned}
$$
In the following steps, we repeatedly apply the result of \Cref{lem: stability fully}, which was 
stated for a forward problem, to the dual problem \eqref{eq: dual problem}, which is defined backwards in time.
The results of the lemma still hold true with the appropriate modifications. Doing so yields for the 
first term the estimate
$$
J_1\le \left(\sum_{m=1}^M \|\na (\psi-\chi_{kh})\|_{I_m\times \Om}^2\right)^{\frac12}
\left(\sum_{m=1}^M \|\na \partial_t z_{kh}\|_{I_m \times \Om}^2\right)^{\frac12}
\le C\|\na (\psi-\chi_{kh})\|_{I\times \Om} \|\na (\psi_{kh}-\chi_{kh})\|_{I\times \Om}.
$$
To estimate $J_2$, we use the definition of the operators  $R_h$ and $A_h$, and 
\Cref{lem: stability fully} applied to $z_{kh}$ to obtain 
$$
\begin{aligned}
  J_2 &= \int_I a_h(R_h \psi -\chi_{kh},z_{kh})dt
  =\int_I (\na (R_h \psi - \chi_{kh}),\na A_hz_{kh})_\Om dt\\
      &\le \|\na (R_h \psi - \chi_{kh})\|_{I\times \Om}\|\na A_hz_{kh}\|_{I\times \Om} 
      \le C\|\na (R_h \psi - \chi_{kh})\|_{I\times \Om}\|\na (\psi_{kh}-\chi_{kh})\|_{I\times \Om}.
\end{aligned}
$$
Similarly, using the definition of the operator $\pi_k$, we have 
$$
\begin{aligned}
  J_3 &= \sum_{m=1}^M(\na(\pi_k \psi(t_m)-\chi_{kh}(t_m)),[\na z_{kh}]_m)_\Omega\\
      &\le  \sum_{m=1}^M\|\na (\pi_k  \psi(t_m) - \chi_{kh}(t_m))\|_\Omega\|[\na z_{kh}]_m\|_\Omega\\
      &\le  \left(\sum_{m=1}^M k_m\|\na(\pi_k \psi - \chi_{kh})\|^2_{L^\infty(I_m; L^2(\Om))}\right)^{1/2}\left(\sum_{m=1}^M k^{-1}_m\|[\na z_{kh}]_m\|^2_\Omega\right)^{1/2}\\
      &\le  C\left(\sum_{m=1}^M \|\na(\pi_k \psi -\chi_{kh})\|^2_{I_m\times\Omega}\right)^{1/2}\left(\sum_{m=1}^M k^{-1}_m\|[\na z_{kh}]_m\|^2_\Omega\right)^{1/2}\\
      &\le C\|\na (\pi_k \psi -\chi_{kh})\|_{I\times \Om}\|\na (\psi_{kh} - \chi_{kh})\|_{I\times \Om},
\end{aligned}
$$
where we used the one-dimensional inverse inequality on $I_m$.
Combining the estimates for $J_1$, $J_2$, and $J_3$ and canceling $\|\na( \psi_{kh} - \chi_{kh})\|_{I\times \Om}$ on both sides, we obtain the estimate \eqref{eq:error_fully_discrete_chikh}.  
From this, we obtain by inserting twice $\pm \psi$ and applying triangle inequality, that
\begin{equation*}
    \norm{\na(\psi-\psi_{{kh}})}_{I\times\Omega} 
    \leq C  \Big(\norm{\na(\psi-\chi_{kh})}_{I\times\Omega}
    +\norm{\na (R_h \psi - \psi)}_{I\times\Omega} 
    +\norm{\na (\pi_k\psi - \psi)}_{I\times\Omega}
    +2\norm{\na (\psi - \chi_{kh})}_{I\times\Omega}
    \Big).
\end{equation*}
Collecting the $\norm{\na (\psi - \chi_{kh})}_{I\times\Omega}$ terms together then concludes the proof.
\end{proof}

\subsection{Specific spatial finite element spaces}

Here, we provide several examples of finite element spaces for solving the stationary biharmonic problem. 

\subsubsection{Conformal $C^1$ elements}\label{sec: C1_method}

It is natural to seek a finite element space $V_h\subset H^2_0(\Om)$, i.e., a conformal finite element approximation.  
In two dimensions, there are a number of such constructions, which usually require rather high polynomial order, we will focus only on $C^1$ Argyris elements, (cf. \cite[p.~273]{Ciarlet_1991} for other finite element methods for fourth order problems). Trivially, \Cref{ass:Vh_contained} is satisfied with $\theta=1$ and \Cref{ass:discrete_bilin_form} holds since the method is conformal and the discrete bilinear form coincides with the continuous one, see \Cref{remark: assumptions}. In the case of $\Om$ being convex, the standard duality argument and a priori error estimates show   
\begin{equation}\label{eq: error in H1 norm C1}
\|\psi-R_h \psi\|_{H^1(\Om)}\le Ch^2\|\psi\|_{H^{3}(\Om)}.
\end{equation}

\subsubsection{The $C^0$ interior penalty method}\label{sec: C0_method}
This method is very attractive since it uses Lagrange elements of degree $l\geq 2$.
To define the method, we need some additional notation. Let $\mce$ be the set of edges in $\mathcal{T}_h$.
For $e\in\mce$ and a suitably regular function $v$
we define the jump  $\jump{\frac{\pa v}{\pa n}}$ of the normal derivative of $v$ across an edge $e$  and the average of the second normal derivative $\avg{\frac{\pa^2 v}{\pa n^2}}$ of $v$ on an edge $e$  as follows. If $e\subset \Om$, 
we take $n_e$ to be one of the two unit vectors normal to $e$. Then $e$ is the common side of two triangles $T_+\in \mathcal{T}_h$ and $T_{-}\in \mathcal{T}_h$, where $n_e$ is pointing from $T_{-}$ to $T_{+}$.
Thus, on such $e$ we define
$$
\jump{\frac{\pa v}{\pa n}} = \frac{\pa v_{T_+}}{\pa n}\mid_e-\frac{\pa v_{T-}}{\pa n}\mid_e\quad\text{and}\quad 
\avg{\frac{\pa^2 v}{\pa n^2}} = \frac12 \left( \frac{\pa^2 v_{T_+}}{\pa n^2}\mid_e+\frac{\pa^2 v_{T-}}{\pa n^2}\mid_e\right).
$$
We note that the above definitions do not depend on the choice of $n_e$.
If $e\subset \pa\Om$, we take $n_e$ to be the unit normal pointing outside $\Om$ and
$$
\jump{\frac{\pa v}{\pa n}} =-\frac{\pa v}{\pa n_e}\quad \text{and}\quad\avg{\frac{\pa^2 v}{\pa n^2}} = \frac{\pa^2 v}{\pa n_e^2}.
$$
It is clear that $\jump{\frac{\pa v}{\pa n}},\avg{\frac{\pa^2 v}{\pa n^2}}$ are well defined for 
$v \in V_h$. Moreover, following the discussion in \Cref{remark: assumptions}, 
the space $X$ from \Cref{ass:discrete_bilin_form} is of the form $X = H^2_0(\Om) \cap H^{\frac52 +\varepsilon}(\Om)$ for some $\varepsilon >0$. Thus all considered jumps and averages are well defined for $v \in X$.
We thus define the bilinear form $a_h:(X+V_h)\times (X+V_h)\to \mathbb{R}$ by
\begin{equation}\label{eq:interior_penalty_ah}
\begin{aligned}
a_h(v,w)
&= \sum_{T\in \mct} \int_T D^2 v:D^2 w\, dx\\
& \quad +\sum_{e\in \mce}\int_e \left(\avg{\frac{\pa^2 v}{\pa n^2}}\jump{ \frac{\pa w}{\pa n}}
 +\jump{\frac{\pa v}{\pa n}}\avg{\frac{\pa^2 w}{\pa n^2}}
+\frac{\eta}{|e|} \jump{\frac{\pa v}{\pa n}}\jump{\frac{\pa w}{\pa n}}\right)ds,
\end{aligned}
\end{equation}
where 
$$
 D^2 v:D^2 w = \sum_{i,j=1}^2 \pa^2_{ij}v\ \pa^2_{ij} w,
$$
$|e|$ is the length of the edge $e$ and $\eta>0$ is a sufficiently large penalty constant. 
As discussed in \Cref{remark: assumptions}, 
for this choice of discretization, the results of \cite[Lemma 5]{BrennerS_SungLY_2005} show that for any 
polygonal domain $\Omega$, the \Cref{ass:discrete_bilin_form} holds for
$\theta = \frac12 - \varepsilon$, where $\varepsilon>0$ can be chosen arbitrarily small.
It is a standard result, that the \Cref{ass:Vh_contained} is also satisfied.
As mentioned above, the space $X$ in this case is given by $H^2_0(\Om)\cap H^{\frac52+\varepsilon}(\Om)$.
Hence, the Ritz projection $R_h$ introduced in \eqref{eq:def_ritz_projection} is well defined for any 
$\psi \in H^2_0(\Om)\cap H^{\frac52+\varepsilon}(\Om)$.
Moreover, if $\Omega$ is convex, we have the following error estimate 
(cf. \cite{BrennerS_SungLY_2005}, Theorem 5, with $\alpha=1$)
\begin{equation}\label{eq: error in H1 norm}
\|\psi-R_h \psi\|_{H^1(\Om)}\le Ch^2\|\psi\|_{H^{3}(\Om)}.
\end{equation}
Thus, in the above setting, for conformal $C^1$ and $C^0$ interior penalty methods, we obtain the following rates of convergence.
    \begin{theorem}\label{thm:convergence_orders_Czero}
        Let $\Omega$ be convex, $f \in L^2(I;H^{-1}(\Om))$ and $\psi_0 \in H^2_0(\Om)$. 
        Let further $\psi$, $\psi_{kh}$ solve the continuous and fully discrete time-dependent stream-function
        equations \eqref{eq:transient_biharmonic_weak} and \eqref{eq:fully_discrete_instationary}, respectively.
        Then there exists $C>0$ independent of $h$ and $k$ such that 
        \begin{equation*}
            \norm{\nabla (\psi-\psi_{kh})}_{L^2(I;L^2(\Om))} 
            \le C (k + h^2)\left(\norm{f}_{L^2(I;H^{-1}(\Om))} + \norm{\psi_0}_{H^2_0(\Om)}\right).
        \end{equation*}
    \end{theorem}
    \begin{proof}
        By the discussion preceding this theorem, the Assumptions \ref{ass:Vh_contained} and
        \ref{ass:discrete_bilin_form} are fulfilled.
        Hence we need to estimate the terms presented in \Cref{thm:error fully discrete}. By the regularity 
        results derived in \Cref{prop:weak_stability}, 
        the stationary error estimate of the Ritz projection \eqref{eq: error in H1 norm}
        and the standard error estimate for the time projection $\pi_k$,we have 
        \begin{equation}\label{eq:error_time_projection}
            \norm{\na (\pi_k \psi- \psi)}_{L^2(I\times \Om)} 
            \le C k \norm{\pa_t \na \psi}_{L^2(I\times \Om)}
            \le C k \left(\norm{f}_{L^2(I;H^{-1}(\Om))} + \norm{\psi_0}_{H^2_0(\Om)}\right),
        \end{equation}
        \begin{equation}\label{eq:error_Ritz_projection}
            \norm{\na (R_h \psi- \psi)}_{L^2(I\times \Om)} 
            \le C h^2 \norm{\psi}_{L^2(I;H^3(\Om))}
            \le C h^2 \left(\norm{f}_{L^2(I;H^{-1}(\Om))} + \norm{\psi_0}_{H^2_0(\Om)}\right).
        \end{equation}
        In the best approximation term, following the proof of \cite[Theorem 12]{LeykekhmanD_VexlerB_2017a}
        we choose $\chi_{kh} = \Pi_h \pi_k \psi$, where $\Pi_h$ is the 
        $H^1_0$ projection introduced in \eqref{eq:definition_Pih}.
        This, together with the triangle inequality and the stability of
        $\Pi_h$ in $H^1_0(\Om)$, see \eqref{eq:stability_Pih}, yields
        \begin{align*}
            \norm{\na (\psi - \chi_{kh})}_{L^2(I\times \Om)} 
            &\le \norm{\na(\psi - \Pi_h \psi)}_{L^2(I\times \Om)}  
            + \norm{\na(\Pi_h(\psi - \pi_k \psi))}_{L^2(I\times \Om)}\\
            &\le \norm{\na(\psi - \Pi_h \psi)}_{L^2(I\times \Om)}  
            + \norm{\na(\psi - \pi_k \psi)}_{L^2(I\times \Om)}.
        \end{align*}
        We have already estimated the second term in \eqref{eq:error_time_projection}, and it is a 
        classical result from numerical analysis of the Poisson problem, that due to $l\ge 2$ and the
        regularity of $\psi$, the projection $\Pi_h$ satisfies the error estimate
        \begin{equation*}
            \norm{\na(\psi - \Pi_h \psi)}_{L^2(I\times \Om)} \le C h^2 \norm{\psi}_{L^2(I;H^3(\Om))}
            \le C h^2 (\norm{f}_{L^2(I;H^{-1}(\Om))} + \norm{\psi_0}_{H^2_0(\Om)}),
        \end{equation*}
        which concludes the proof.
    \end{proof}

Naturally, $\Curl(\psi_{kh})=(\partial_2 \psi_{kh}, -\partial_1 \psi_{kh})^T$ is discontinuous across triangulation boundaries. 
One way to address this issue is postprocessing \cite{BrennerS_SungLY_2005}, but this topic lies outside of the scope of this paper.

\section{Numerical results}\label{sec:numerical_experiments}
To demonstrate the theoretical error estimates derived in this work, we generate an example where the analytic 
solution $\uu$ to the time-dependent Stokes equations is known.
We choose $I = (0,1]$, $\Omega = [0,1]^2$ and calculate 
$\ww = \Curl(\Phi)$, where $\Phi\colon\R^2 \to \R, \Phi(x) = \sin(2\pi x_1)^2 \sin(2\pi x_2)^2$.
We then choose $\uu = \sin(2 \pi t)\ww$ and set
\begin{equation*}
    \g = \pa_t \uu - \Delta \uu.
\end{equation*}
By construction $\uu$ satisfies $\nabla \cdot \uu = 0$, $\uu(0) = \oo$ and $\uu|_{\pa \Omega} = 0$,
and thus solves the Stokes equations with right-hand side $\g$, together with an associated pressure $p=0$.
We first discretize the standard weak formulation of these equations by dG(0) in time and the MINI element
in space. 
The numerically observed errors for this setting are reported in \Cref{fig:error_analysis}.
\revB{ In the left-hand plot, we have set $h=2^{-8}$ constant and refined in time, 
  whereas in the right-hand plot, we have chosen $k=2^{-10}$ and have refined in space only.
}
We then proceed by modifying the right-hand side $\g$ by a gradient field, i.e., set the new right-hand side
to
\begin{equation*}
    \tilde \g = \g + 10^5 \begin{pmatrix} x_1^{-0.49} \\ 0 \end{pmatrix} \in L^\infty(I;L^2(\Om)^2).
\end{equation*}
As $\tilde \g$ and $\g$ only differ by a gradient field, they yield the same velocity solution $\uu$, 
and only the pressure changes to \revB{$\tilde p = p + 10^5 \cdot 0.51^{-1} \cdot x_1^{0.51}$}.
As for the discretization of the Stokes with mixed finite elements,
the velocity error is bounded by the sum of best approximation errors for 
velocity and pressure, one can observe in \Cref{fig:error_analysis}, that the error in the velocity increases.
\revB{In the left-hand plot this effect is only visible for fine $k$, since for coarse $k$,
  the discretization error in time dominates the additional spatial error.
  In the convergence plot with respect to $h$ on the right-hand side,
  the increase in the error can be observed at all discretization levels.
}
Lastly we set up a discrete stream-function formulation by taking 
$f = -\curl(\g)$ and solving:
\revB{Find $\psi_{kh} \in X^r_k(V_h)$ such that}
\begin{equation*}
    B_h(\psi_{kh},\varphi_{kh}) = \IOprod{f,\varphi_{kh}} \quad \textforall \varphi_{kh} \in X^r_k(V_h).
\end{equation*}
\revB{Analogous to the mixed discretization above, we use dG(0) in time, i.e., use $r=0$, but 
    in space we now use Lagrange finite elements of degree 2:
    $$V_h = \set{v_h \in C(\bar \Omega): v_h\vert_{K} \in \mathbb P_2(K) \quad \forall K \in \mathcal T_h, \quad v_h\vert_{\pa \Om} = 0}.$$
    The bilinear form $a_h(\cdot,\cdot)$ used in the definition of $B_h$ is in this case given by the interior
    penalty formulation \eqref{eq:interior_penalty_ah}.
}
Again we present the error $\uu - \Curl(\psi_{kh})$ in \Cref{fig:error_analysis}. 
Note that we observe exactly the error estimates derived in \Cref{thm:convergence_orders_Czero}.
Even though the absolute values of the errors are higher than the error for the MINI discretization
applied to $\g$, we obtain the same error values, if we applied the stream-function formulation to 
$\tilde \g$, as in the setting of this problem, $\curl$ is applied to the right-hand side, and the 
added gradient field vanishes in the right-hand side for $\psi_{kh}$.
Thus the same convergence behaviour is obtained for the modified right-hand side $\tilde \g$,
showing the pressure-robustness of this method.

% \begin{figure}[h]
%     \centering
%     \includegraphics[width=.8\linewidth]{plots/Stokes_Convergence_analytic.png}
%     \caption{Convergence for discretization with MINI element}
% %     \label{fig:mini_convergence}
% \end{figure}

\begin{figure}[h]
    \hspace{-20mm}
  \begin{minipage}{.40\linewidth}
  \begin{tikzpicture}[]
    \begin{axis}[legend pos=north west, legend style={draw=none},
                  xmin=0.001, xmax=0.1,
                  xmode = log, ymode = log,
                  xlabel={$k$},
                  ymin=0.01,ymax=1
                  ]
      \addplot[black,mark=star] table[x=k,y=error,col sep=comma]{evaluations/Stokes_kConvergence_MINI_analytic_instationary_v10.csv};
      \addplot[red,mark=diamond] table[x=k,y=error,col sep=comma]{evaluations/Stokes_kConvergence_MINI_analytic_instationary_GradField_300725.csv};
      \addplot[blue,mark=o] table[x=k,y=error,col sep=comma]{evaluations/Bilap_kConvergence_analytic_instationary_v10.csv};
      \addplot[black,dashed] table[x=k,y expr=10*x,col sep=comma]{evaluations/Stokes_kConvergence_MINI_analytic_instationary_v10.csv};
      \legend{{MINI ($\g$)},{MINI ($\tilde \g$)},{streamfct.},{$O(k)$}}
    \end{axis}
  \end{tikzpicture}
  \end{minipage}
  \hspace{15mm}
  \begin{minipage}{.40\linewidth}
  \begin{tikzpicture}[]
    \begin{axis}[legend pos=north west, legend style={draw=none},
                  xmode = log, ymode = log,
                  xmin=0.001, xmax=0.1,
                  xlabel={$h$},
                  ymin=0.001,ymax=10
                  ]
      \addplot[black,mark=star] table[x=h,y=error,col sep=comma]{evaluations/Stokes_hConvergence_MINI_analytic_instationary_v10.csv};
      \addplot[red,mark=diamond] table[x=h,y=error,col sep=comma]{evaluations/Stokes_hConvergence_MINI_analytic_instationary_GradField_300725.csv};
      \addplot[blue,mark=o] table[x=h,y=error,col sep=comma]{evaluations/Bilap_hConvergence_analytic_instationary_v10.csv};
      \addplot[black,dashed] table[x=h,y expr=700*x*x,col sep=comma]{evaluations/Stokes_hConvergence_MINI_analytic_instationary_v10.csv};
      \legend{{MINI ($\g$)},{MINI ($\tilde \g$)},{streamfct.},{$O(h^2)$}}
    \end{axis}
  \end{tikzpicture}
  \end{minipage}
  \centering
    \caption{Convergence of discretization with MINI element for different right-hand sides (black/red) and 
    interior penalty method for stream-function formulation (blue).}
    \label{fig:error_analysis}
\end{figure}

\revB{
\begin{figure}[h]
  \begin{minipage}{.40\linewidth}
  \begin{tikzpicture}[]
    \begin{axis}[legend pos=north west, legend style={draw=none},
                  xmin=0.0001, xmax=0.1,
                  xmode = log, ymode = log,
                  xlabel={$k=ch^2$},
                  ymin=0.001,ymax=10
                  ]
      \addplot[black,mark=star] table[x=k,y=error,col sep=comma]{evaluations/Stokes_khConvergence_MINI_analytic_instationary_revision.csv};
      \addplot[red,mark=diamond] table[x=k,y=error,col sep=comma]{evaluations/Stokes_khConvergence_MINI_analytic_instationary_revision_GradField.csv};
        \addplot[blue,mark=o] table[x=k,y=error,col sep=comma]{evaluations/Bilap_khConvergence_analytic_instationary_revision.csv};
      \addplot[black,dashed] table[x=k,y expr=10*x,col sep=comma]{evaluations/Bilap_khConvergence_analytic_instationary_revision.csv};
      %\legend{{streamfct.},{$O(k)$}}
      \legend{{MINI ($\g$)},{MINI ($\tilde \g$)},{streamfct.},{$O(k+h^2)$}}
    \end{axis}
  \end{tikzpicture}
  \end{minipage}
  \centering
    \caption{Convergence rates for simultaneous space-time refinement of the
        MINI element for different right-hand sides (black/red) and 
    interior penalty method for stream-function formulation (blue).}
    \label{fig:error_simult_refinement}
\end{figure}
  The reduction of the convergence rate with respect to $h$ at fine discretizations for both the MINI discretization
  and the interior penalty approach can be explained by the time error dominating in this regime.
  Due to the first order term $O(k)$ compared to the second order $O(h^2)$, this indicates that 
  the chosen $k=2^{-10}$ is still too coarse, to observe the asyptotic convergence.
  To eliminate the possibility, that the effect is caused by other reasons, in \Cref{fig:error_simult_refinement} 
  $k$ and $h$ are refined simultaneously, with a coupling condition of $k = c h^2$.
  This leads to a theoretical convergence rate of $O(k+h^2) = O(h^2)$ which agrees with the numerically observed errors.
%   The results, as well as the estimated-orders-of-convergence w.r.t. $h$ are reported in 
%   \Cref{tab:convergence_simult_refinement}.
%   The estimated-orders-of-convergence, are calculated as
%   \begin{equation*}
%     (\mathrm{EOC}_h)_{j} = 
%     % \frac{\log (\mathrm{err}(h_j)) - \log (\mathrm{err}(h_{j-1}))}{\log(h_j) - \log(h_{j-1})}.
%     \frac{\log (\norm{\nabla (\psi-\psi_{k_j h_j})}_{L^2(I;L^2(\Om))}) 
%       -\log (\norm{\nabla (\psi-\psi_{k_{j-1}h_{j-1}})}_{L^2(I;L^2(\Om))})}{\log(h_j) - \log(h_{j-1})}.
%   \end{equation*}
% \begin{table}[h]
%   \centering
% \begin{tabular}{rlllll}
%   \hline
%   \rule{0pt}{1.2em} $h$ & $2^{-4}$ & $2^{-5}$ &$2^{-6}$ &$2^{-7}$ &$2^{-8}$\\
%   $k$ & $2^{-5}$ & $2^{-7}$ &$2^{-9}$ &$2^{-11}$ & $2^{-13}$ \\
%   error &5.05e-01 & 1.54e-01 & 4.03e-02 & 1.01e-02 & 2.50e-03\\
%   %error &5.04627552e-01 & 1.53550032e-01 & 4.03483794e-02 & 1.01086686e-02 & 2.50495931e-03\\
%   $\mathrm{EOC}_h$ & - & 1.72 & 1.93& 2.00 & 2.01\\
%   % $\mathrm{EOC}_h$ & - & 1.726510164741479 & 1.9281261793025404& 1.996917734350147 & 2.0127339203005934\\
%   \hline
%   \vspace{1mm}
% \end{tabular}
% \\ \textcolor{red}{Decide on additional figure vs table}
% \caption{Convergence of the interior penalty method for simultaneous refinement of $k$ and $h$, using $k = c h^2$.}
% \label{tab:convergence_simult_refinement}
% \end{table}
}

\appendix
\section{Sharpness of the time regularity. Counterexample}\label{sec:appendix}
In \Cref{thm:existence_uniqueness_standard_regularity} we especially proved that for $f \in L^2(I;H^{-2}(\Omega))$ and $\psi_0 \in H^1_0(\Omega)$ the solution $\psi$ of \eqref{eq:transient_biharmonic_weak} possesses the regularity $\partial_t \Delta \psi \in L^2(I;H^{-2}(\Omega))$ and that the estimate
\[
\|\partial_t \Delta \psi\|_{L^2(I;H^{-2}(\Omega))} 
+ \|\psi\|_{L^2(I;H^2_0(\Omega))} \le C\left(\|f\|_{L^2(I;H^{-2}(\Om))}+\|\na \psi_0\|_{L^2(\Om)}\right)
\]
holds. The constant $C$ here depends only on the domain $\Omega$. In this section, we demonstrate the sharpness of this result in the following sense: We provide a counterexample showing that such an estimate is not possible for $\|\partial_t \psi\|_{L^2(I;L^2(\Omega))}$ and not even for $\|\partial_t \psi\|_{L^2(I;H^{-1}(\Omega))}$. The latter regularity is claimed in \cite[Theorem 2.3]{KimQuasiGeostrophic2025} for a version of \eqref{eq:transient_biharmonic_weak} that includes certain nonlinear terms, under the same assumptions on the data as in \Cref{thm:existence_uniqueness_standard_regularity}.

We consider $I=(0,1)$, $\Omega=(0,1)$ and set
\begin{equation}\label{eq:def_psi_n}
\psi_n(t,x) := \sum_{k=1}^n a_{k,n}(t)\,\big(\cos(2\pi k x)-1\big),\qquad (t,x)\in I\times\Omega
\end{equation}
with the coefficients $a_{k,n}(t)$ to be chosen later. By construction, every $\psi_n$ fulfills the boundary conditions
\[
 \psi_n(t,0)=\psi_n(t,1)=\partial_x\psi_n(t,0)=\partial_x\psi_n(t,1)=0\quad (t\in I)
\]
and we have
\[
\partial_t\psi_n(t,x) = \sum_{k=1}^n a'_{k,n}(t)\big(\cos(2\pi k x)-1\big).
\]
We set the corresponding sequence of right-hand sides as
\begin{equation}\label{eq:app_def_fn}
f_n := -\partial_t\partial_x^2\psi_n + \partial_x^4\psi_n.
\end{equation}
Defining $\lambda_k = (2\pi k)^2$ we obtain by a direct calculation
\[
	f_n(t,x)=\sum_{k=1}^n \big(\lambda_k a'_{k,n}(t)+\lambda_k^2 a_{k,n}(t)\big)\,\cos(2\pi k x).
\]
In the following lemma we calculate certain norms of $f_n$ and $\psi_n$.
\begin{lemma}
Let $S_n(t)$ and $R_n(t)$ be defined as
\begin{equation}\label{eq:def:S_n_R_n}
S_n(t):=\sum_{k=1}^n a'_{k,n}(t) \quad \text{and} \quad R_n(t):=\sum_{k=1}^n \frac{a'_{k,n}(t)}{\lambda_k}.
\end{equation}
Then for the construction above there holds:
\begin{subequations}
\begin{align}
    \|f_n\|^2_{L^2(I;H^{-2}(\Omega))} &= \frac12\int_0^1 \sum_{k=1}^n \big|a'_{k,n}(t)+\lambda_k a_{k,n}(t)\big|^2\,dt,\label{eq:app_repr_1}\\
    \|\partial_t \partial^2_x \psi_n\|^2_{L^2(I;H^{-2}(\Omega))} &= \frac12\int_0^1 \sum_{k=1}^n |a'_{k,n}(t)|^2\,dt, \label{eq:app_repr_2}\\
    \|\partial_t \psi_n\|^2_{L^2(I;L^2(\Omega))} &= \int_0^1\left[\frac12\sum_{k=1}^n |a'_{k,n}(t)|^2 + |S_n(t)|^2\right]\,dt, \label{eq:app_repr_3}\\
    \|\partial_t \psi_n\|^2_{L^2(I;H^{-1}(\Omega))} &= \int_0^1\left[\sum_{k=1}^n \frac{1}{2\lambda_k}|a'_{k,n}(t)|^2 + \frac{1}{12}|S_n(t)|^2+2S_n(t)R_n(t)\right] \,dt. \label{eq:app_repr_4}
\end{align}
\end{subequations}
\end{lemma}
\begin{proof}
In order to calculate the $H^{-2}(\Omega)$ norm of a sum
\[
g(x) = \sum_{k=1}^n g_k\,\cos(2\pi k x)
\]
we solve the equation
\[
\varphi_g^{(4)} = g, \quad \varphi_g(0) = \varphi_g(1) = \varphi_g'(0) = \varphi_g'(1) = 0
\]
and obtain (cf. \Cref{lemm:h2_norm_equivalence})
\[
\|g\|_{H^{-2}(\Omega)} = \|\varphi''_g\|_{L^2(\Omega)}.
\]
A direct calculation gives
\[
\varphi''_g = \sum_{k=1}^n \frac{g_k}{\lambda_k}\,\cos(2\pi k x)
\]
and therefore by orthogonality of $\cos(2\pi k x)$ with respect to $L^2(\Omega)$ and the fact that $\norm{\cos(2\pi k x)}^2_{L^2(\Omega)} = \frac{1}{2}$ we get
\[
\|g\|_{H^{-2}(\Omega)}^2 = \half \sum_{k=1}^n \frac{g_k^2}{\lambda_k^2}.
\]
Applying this to $f_n(t,\cdot)$ given by \eqref{eq:app_def_fn} and integrating in time we obtain the first representation \eqref{eq:app_repr_1}.\\

Similarly, using
\[
 \partial_t\partial_x^2\psi_n(t,x) = -\sum_{k=1}^n \lambda_k a'_{k,n}(t)\,\cos(2\pi k x)
\]
we get \eqref{eq:app_repr_2}.
% \[
% 	\|\partial_t \partial^2_x \psi_n\|^2_{L^2(I;H^{-2}(\Omega))} = \frac12\int_0^1 \sum_{k=1}^n |a'_{k,n}(t)|^2\,dt.
% \]
% 
To obtain the representation for the $L^2$-norm of $\partial_t \psi_n$ of \eqref{eq:app_repr_3}, we rewrite
\[
\partial_t\psi_n(t,x) = \sum_{k=1}^n a'_{k,n}(t)\cos(2\pi k x) - S_n(t)
\]
and get, using the orthogonality of the $\cos$-terms and $(\cos(2\pi k x),1)_{L^2(\Omega)}=0$,
the appropriate expression.\\

To calculate the $H^{-1}(\Omega)$ norm of a sum
\[
g(x) = \sum_{k=1}^n g_k\left(\cos(2\pi k x)-1\right)
\]
we solve the equation
\[
-w''_g = g, \quad w_g(0) = w_g(1) = 0
\]
and obtain
\[
\|g\|_{H^{-1}(\Omega)} = \|w'_g\|_{L^2(\Omega)}.
\]
A direct calculation shows
\[
w'_g = \sum_{k=1}^n g_k \left(-\frac{1}{\sqrt{\lambda_k}}\sin(2\pi k x) + x -\frac{1}{2} \right)
\]
and for $g = \partial_t \psi_n$, i.e., choosing $g_k = a'_{k,n}(t)$, we get
\[
w'_g = -\sum_{k=1}^n \frac{a'_{k,n}(t)}{\sqrt{\lambda_k}}\sin(2\pi k x) + \left(x -\frac{1}{2} \right)S_n(t).
\]
Using the $L^2(\Omega)$ orthogonality of the $\sin$-terms and the fact that
\[
\left(\sin(2\pi k x),x-\frac{1}{2}\right)_{L^2(\Omega)} = -\frac{1}{\sqrt{\lambda_k}} \quad\text{and}\quad \left\|x-\frac{1}{2}\right\|_{L^2(\Omega)}^2 = \frac{1}{12}
\]
we obtain
\[
\|\partial_t \psi_n(t,\cdot)\|^2_{H^{-1}(\Omega)} = \frac{1}{2}\sum_{k=1}^n \frac{1}{\lambda_k} |a'_{k,n}(t)|^2 + \frac{1}{12} S_n(t)^2 + 2 S_n(t) \sum_{k=1}^n \frac{a'_{k,n}(t)}{\lambda_k}.
\]
By integrating in time and applying the definition of $R_n(t)$, we obtain the last representation 
\eqref{eq:app_repr_4}.
\end{proof}

As the next step we make a special choice for the coefficients $a_{k,n}(t)$, $k=1,2,\dots,n$ in \eqref{eq:def_psi_n}. We determine $a_{k,n}$ as the unique solution of the ODE
\[
 a'_{k,n}(t)+\lambda_k a_{k,n}(t)=r_n(t),\qquad a_{k,n}(0)=0.
\]
with 
\[
r_n(t) = n \chi_{[0,\varepsilon_n]}, \; \varepsilon_n = n^{-3}.
\]
By construction we get $\psi_n(0,x) =0$.

We directly obtain
\[
a_{k,n}(t) = 
\begin{cases}
	\frac{n}{\lambda_k}(1-e^{-\lambda_k t}), & 0 \le t \le \varepsilon_n\\
	\frac{n}{\lambda_k}(1-e^{-\lambda_k \varepsilon_n})e^{-\lambda_k(t-\varepsilon_n)}, & \varepsilon_n \le t \le 1
\end{cases}
\]
and thus
\[
a'_{k,n}(t) = 
\begin{cases}
	ne^{-\lambda_k t}, & 0 \le t < \varepsilon_n\\
	-n(1-e^{-\lambda_k \varepsilon_n})e^{-\lambda_k(t-\varepsilon_n)}, & \varepsilon_n < t \le 1
\end{cases}
.
\]
Please note that we have
\[
a'_{k,n}(t)  \
\begin{cases}
	>0, & 0 \le t < \varepsilon_n\\
	<0, & \varepsilon_n < t \le 1
\end{cases}
\]
for every $k=1,2,\dots,n$. Thus, for $S_n(t)$ and $R_n(t)$ defined in \eqref{eq:def:S_n_R_n} there holds
\begin{equation}\label{eq:SR_positive}
S_n(t) R_n(t) >0 \quad \text{for all } t \neq \varepsilon_n.
\end{equation}

\begin{theorem}
For the construction above there holds
\begin{subequations}
\begin{align}
        \|f_n\|^2_{L^2(I;H^{-2}(\Omega))} &= \frac{1}{2} \label{eq:app_est_1}\\
        \|\partial_t \partial^2_x \psi_n\|^2_{L^2(I;H^{-2}(\Omega))} &\le \frac{1}{2} +\pi^2 \label{eq:app_est_2}\\
        \|\partial_t \psi_n\|^2_{L^2(I;L^2(\Omega))} & \ge cn \label{eq:app_est_3}\\
        \|\partial_t \psi_n\|^2_{L^2(I;H^{-1}(\Omega))} &\ge cn \label{eq:app_est_4}
\end{align}
\end{subequations}
with a constant $c>0$ independent of $n$.
\end{theorem}
\begin{proof}
By construction there holds
\[
\|f_n\|^2_{L^2(I;H^{-2}(\Omega))} = \frac{1}{2} \int_0^1 \sum_{k=1}^n r_n^2(t)\, dt = \frac{n}{2} \int_0^1 r_n^2(t)\, dt = \frac{n}{2} \int_0^{\varepsilon_n} n^2 \,dt = \frac{n^3 \varepsilon_n}{2} = \frac{1}{2}.
\]
This provides the identity \eqref{eq:app_est_1}. For the second estimate \eqref{eq:app_est_2} we 
use the representation from \eqref{eq:app_repr_2} and obtain by splitting the integral
\begin{equation}\label{eq:app_split_integral}
\|\partial_t \partial^2_x \psi_n\|^2_{L^2(I;H^{-2}(\Omega))} =  \frac12\int_0^{\varepsilon_n} \sum_{k=1}^n |a'_{k,n}(t)|^2\,dt + \frac12\int_{\varepsilon_n}^1 \sum_{k=1}^n |a'_{k,n}(t)|^2\,dt.
\end{equation}
For the first sum of \eqref{eq:app_split_integral} we observe
\[
 |a'_{k,n}(t)| = n e^{-\lambda_k t} \le n
\]
and thus
\[
\frac12\int_0^{\varepsilon_n} \sum_{k=1}^n |a'_{k,n}(t)|^2\,dt \le \frac12 n \varepsilon_n n^2 = \frac12.
\]
To estimate the second sum in \eqref{eq:app_split_integral}, we calculate the integral and apply the inequality $1-e^{-x} \le x $ for $x=\lambda_k \varepsilon _n$ resulting in 
  \begin{align*}
	\int_{\varepsilon_n}^1 |a'_{k,n}(t)|^2 dt
	& =n^2(1-e^{-\lambda_k \varepsilon_n})^2 \int_{\varepsilon_n}^1 e^{-2\lambda_k(t-\varepsilon_n)} dt
	= \frac{n^2(1-e^{-\lambda_k \varepsilon_n})^2}{2 \lambda_k} [1 - e^{-2\lambda_k(1-\varepsilon_n)}]\\
	& \le \frac{n^2(\lambda_k \varepsilon_n)^2}{2 \lambda_k} = \frac12 n^{-4} \lambda_k = 2 \pi^2 n^{-4} k^2.
\end{align*}
Thus there holds 
\begin{equation*}
	\frac12 \sum_{k=1}^n \int_{\varepsilon_n}^1 |a'_{k,n}(t)|^2 dt \le \frac12 \sum_{k=1}^n 2 \pi^2 n^{-4} k^2
	= \pi^2 n^{-4} \sum_{k=1}^n k^2 \le \pi^2 n^{-4} n n^2 = \pi^2 n^{-1} \le \pi^2.
\end{equation*}
Putting terms together we obtain \eqref{eq:app_est_2}.\\

By the representation of $\|\partial_t \psi_n\|^2_{L^2(I;L^2(\Omega))}$ from \eqref{eq:app_repr_3}, we obtain
\[
\|\partial_t \psi_n\|^2_{L^2(I;L^2(\Omega))} \ge \int_0^1 S_n(t)^2 \ge \int_0^{\varepsilon_n} S_n(t)^2\, dt. 
\]
Moreover, for $\|\partial_t \psi_n\|_{L^2(I;H^{-1}(\Om))}$  we get with the representation \eqref{eq:app_repr_4}
and $\int_0^1 S_n(t) R_n(t) dt \ge 0$ due to \eqref{eq:SR_positive}, 
\[
\|\partial_t \psi_n\|^2_{L^2(I;H^{-1}(\Omega))} \ge  \frac{1}{12}\int_0^1 S_n(t)^2 \ge \frac{1}{12} \int_0^{\varepsilon_n} S_n(t)^2\, dt. 
\]
Thus, in order to show the last two statements \eqref{eq:app_est_3} and \eqref{eq:app_est_4},
it remains to estimate $\int_0^{\varepsilon_n} S_n(t)^2\, dt$ from below.
There holds for $0\le t \le \varepsilon_n$ and $\lambda_k \le \lambda_n$ with $e^{-x} \ge 1-x$ that
\[
e^{-\lambda_k t} \ge 1 - \lambda_k t \ge  1 - \lambda_n  \varepsilon_n  = 1-4 \pi^2 n^{-1}\ge \frac{1}{2}
\]
for $n$ large enough. Thus,
\[
S_n(t) = \sum_{k=1}^n a'_{k,n}(t) \ge n \sum_{k=1}^n e^{-\lambda_k t} \ge \frac{n^2}{2}
\]
for $0\le t \le \varepsilon_n$ and $n$ large enough. This results in
\[
 \int_0^{\varepsilon_n} S_n(t)^2\, dt \ge \frac{n^4}{4} \varepsilon_n = \frac{n}{4}.
\]
This completes the proof.
\end{proof}

The above considerations show that for $f \in L^2(I,H^{-2}(\Omega))$ and $\psi_0 = 0$ estimates like
\[
\|\partial_t \psi\|_{L^2(I;L^2(\Omega))} \le C \|f\|_{L^2(I,H^{-2}(\Omega))} \quad \text{or} \quad \|\partial_t \psi\|_{L^2(I;H^{-1}(\Omega))} \le C \|f\|_{L^2(I,H^{-2}(\Omega))}
\]
cannot hold.
\bibliography{stream_references}
\bibliographystyle{siam}

\end{document}